\documentclass[draft,11pt]{article}

\input xy
\xyoption{all}
\usepackage{amsmath}
\usepackage{amsfonts}
\usepackage{syntonly}
\usepackage{amsthm}
\usepackage{latexsym}

\newcommand{\st}{\mbox{such that}}
\newcommand{\ses}{short exact sequence}
\newcommand{\dses}{double \ses}
\newcommand{\comdia}{$3\times 3$ commutative diagram}

\newcommand{\case}[1]{\textsl{Case} {#1}.}
\newcommand{\Hmat}[2]{\left(\begin{smallmatrix}{#1}&{#2}\end{smallmatrix}\right)}
\newcommand{\Vmat}[2]{\left(\begin{smallmatrix}{#1}\\{#2}\end{smallmatrix}\right)}
\newcommand{\Smat}[4]{\left(\begin{smallmatrix}{#1}&{#2}\\{#3}&{#4}\end{smallmatrix}\right)}

\newcommand{\cat}[1]{\textbf{\ensuremath{\mathcal{#1}}}}

\swapnumbers
\newtheorem{Def}{Definition}[subsection]
\newtheorem{Thm}[Def]{Theorem}
\newtheorem{Lem}[Def]{Lemma}
\newtheorem{MnLem}[Def]{Key Lemma}
\newtheorem{Cor}[Def]{Corollary}
\newtheorem{Cons}[Def]{Construction}

\newtheoremstyle{GKWThm}        
   {0.5cm}                      
   {0 cm}                       
   {}                           
   {}                           
   {\bf}                        
   {.}                          
   {7pt}                        
   {}                           

\newcommand{\UP}{\ar@<0.5ex>}
\newcommand{\DW}{\ar@<-0.5ex>}
\newcommand{\UPON}{\ar@<0.5ex>@{->>}}
\newcommand{\DWON}{\ar@<-0.5ex>@{->>}}

\newcommand{\UPIN}{\ar@<0.5ex>@{>->}}
\newcommand{\DWIN}{\ar@<-0.5ex>@{>->}}
\newcommand{\IN}{\ar@{>->}}

\newcommand{\UPIND}{\ar@<0.5ex>@{>->}!<0pt,-11pt>;}
\newcommand{\DWIND}{\ar@<-0.5ex>@{>->}!<0pt,-11pt>;}
\newcommand{\IND}{\ar@{>->}!<0pt,-11pt>;}

\newcommand{\UPINT}{\ar@<0.5ex>@{>->}!<0pt,11pt>;}
\newcommand{\DWINT}{\ar@<-0.5ex>@{>->}!<0pt,11pt>;}
\newcommand{\INT}{\ar@{>->}!<0pt,11pt>;}

\newcommand{\UPINR}{\ar@<0.5ex>@{>->}!<11pt,0pt>;}
\newcommand{\DWINR}{\ar@<-0.5ex>@{>->}!<11pt,0pt>;}
\newcommand{\INR}{\ar@{>->}!<11pt,0pt>;}

\newcommand{\ON}{\ar@{ ->>}}

\newcommand{\TOr}[1][]{\xymatrix@1{\ar[r]^{#1}&}} 

\newcommand{\onto}[1][]{\xymatrix@1{\ar@{>>}[r]^{#1}&}} 
\newcommand{\into}[1][]{\xymatrix@1{\ar@{^{(}->}^{#1}&}} 
\newcommand{\SES}[3]{\xymatrix{0 \ar[r] & {#1} \ar[r] & {#2} \ar[r] & {#3}
    \ar[r] & 0}}  
\newcommand{\SESM}[5]{\xymatrix{0 \ar[r] & {#1} \ar[r]^{#4} & {#2} \ar[r]^{#5} & {#3}
    \ar[r] & 0}}
\newcommand{\DSES}[7]{\xymatrix{0 \ar[r] & {#1}
    \ar@<0.5ex>[r]^{#4}\ar@<-0.5ex>[r]_{#5} & {#2}
    \ar@<0.5ex>[r]^{#6}\ar@<-0.5ex>[r]_{#7} & {#3} \ar[r] &0}}
\newcommand{\DSE}[3]{\xymatrix{0 \ar[r] & {#1} \ar@<0.5ex>[r]\ar@<-0.5ex>[r]
    & {#2} \ar@<0.5ex>[r]\ar@<-0.5ex>[r] & {#3} \ar[r] &0}}

\newcommand{\Comdia}{\xymatrix{
       a'\UP[r]\DW[r] & a \UP[r]\DW[r] & a''\\
       b'\UP[r]\DW[r]\UP[u]\DW[u] & b\UP[r]\DW[r]\UP[u]\DW[u]& b''\UP[u]\DW[u]\\
       c'\UP[r]\DW[r]\UP[u]\DW[u] & c\UP[r]\DW[r]\UP[u]\DW[u]& c''\UP[u]\DW[u]
       }}

\newcommand{\Phicons}{$\phi$-construction}
\newcommand{\DSESone}[7]{\xymatrix{0 \ar[r] & *+[Fo:<10pt>]{#1}
    \ar@<0.5ex>[r]^{#4}\ar@<-0.5ex>[r]_{#5} & {#2}
    \ar@<0.5ex>[r]^{#6}\ar@<-0.5ex>[r]_{#7} & {#3} \ar[r] &0}}
\newcommand{\DSEStwo}[7]{\xymatrix{0 \ar[r] & {#1}
    \ar@<0.5ex>[r]^{#4}\ar@<-0.5ex>[r]_{#5} & {#2}
    \ar@<0.5ex>[r]^{#6}\ar@<-0.5ex>[r]_{#7} & {#3} \ar[r] &0
    \save {"1,2"+<0pt,10pt>} \PATH ~={**\dir{-}}
                           '"1,3"+<0pt,10pt>
                           `r_d
                           `d_l
                           '"1,2"+<0pt,-10pt>
                           `l_u
                           `u_r
     \restore}}
\newcommand{\DSEStwol}[7]{\xymatrix{0 \ar[r] & {#1}
    \ar@<0.5ex>[r]^{#4}\ar@<-0.5ex>[r]_{#5} & {#2}
    \ar@<0.5ex>[r]^{#6}\ar@<-0.5ex>[r]_{#7} & {#3} \ar[r] &0
    \save {"1,2"+<0pt,10pt>} \PATH ~={**\dir{-}}
                           '"1,3"+<8pt,10pt>
                           `r_d
                           `d_l
                           '"1,2"+<0pt,-10pt>
                           `l_u
                           `u_r
     \restore}}
\newcommand{\DSEStwovar}[7]{\xymatrix{0 \ar[r] & {#1}
    \ar@<0.5ex>[r]^{#4}\ar@<-0.5ex>[r]_{#5} & {#2}
    \ar@<0.5ex>[r]^{#6}\ar@<-0.5ex>[r]_{#7} & {#3} \ar[r] &0
    \save {"1,2"+<0pt,10pt>} \PATH ~={**\dir{-}}
                           '"1,3"+<10pt,10pt>
                           `r_d
                           `d_l
                           '"1,2"+<0pt,-10pt>
                           `l_u
                           `u_r
     \restore}}

\newcommand{\DSESthree}[7]{\xymatrix{0 \ar[r] & {#1}
    \ar@<0.5ex>[r]^{#4}\ar@<-0.5ex>[r]_{#5} & {#2}
    \ar@<0.5ex>[r]^{#6}\ar@<-0.5ex>[r]_{#7} & {#3} \ar[r]
    &0  \save  {"1,2"+<0pt,10pt>} \PATH ~={**\dir{-}}
                           '"1,4"+<0pt,10pt>
                           `r_d
                           `d_l
                           '"1,2"+<0pt,-10pt>
                           `l_u
                           `u_r
    \restore
}}
\newcommand{\SmallRectLeftTopVerFrame}{\save "1,1"+<10pt,0pt> \PATH ~={**\dir{-}}
                  '"2,1"+<10pt,0pt>
                  `d_l
                  `l_u
                  '"1,1"+<-10pt,0pt>
                  `u_r
                  `r_d
              \restore
    }

\newcommand{\SmallRectRightBottomHorFrame}{\save "3,2"+<0pt,10pt> \PATH ~={**\dir{-}}
                  '"3,3"+<0pt,10pt>
                  `r_d
                  `d_l
                  '"3,2"+<0pt,-10pt>
                  `l_u
                  `u_r
              \restore
    }

\newcommand{\Lframe}{\save {"3,3"+<10pt,0pt>}\PATH
    ~={**\dir{-}} `u"3,1" '"3,1"+<20pt,10pt> `l"1,1"+<10pt,0pt>
    '"1,1"+<10pt,0pt> `u^l `l^d '"3,1"+<-10pt,0pt>
    `d^r '"3,3"+<0pt,-10pt> `r^u
    \restore}
\newcommand{\SmallLBLframe}{\save {"3,2"+<10pt,0pt>}\PATH
    ~={**\dir{-}} `u"3,1" '"3,1"+<20pt,10pt> `l"2,1"+<10pt,0pt>
    '"2,1"+<10pt,0pt> `u^l `l^d '"3,1"+<-10pt,0pt>
    `d^r '"3,2"+<0pt,-10pt> `r^u
    \restore}
\newcommand{\SmallRBLframe}{\save {"3,2"+<0pt,10pt>}\PATH
    ~={**\dir{-}} '"3,3"+<-20pt,10pt>
                  `r^u
                  '"2,3"+<-10pt,0pt>
                  `u_r
                  `r_d
                  '"3,3"+<10pt,0pt>
                  `d_l
                  '"3,2"+<0pt,-10pt>
                  `l_u
                  `u_r
    \restore}

\newcommand{\MedChairframe}{\save {"3,3"+<10pt,0pt>}\PATH
    ~={**\dir{-}} '"2,3"+<10pt,0pt>
                  `u^l
                  '"2,1"+<20pt,10pt>
                  `l_u
                  '"1,1"+<10pt,0pt>
                  `u^l
                  `l^d
                  '"3,1"+<-10pt,0pt>
                  `d^r
                  '"3,3"+<0pt,-10pt>
                  `r^u
  \restore}
\newcommand{\Chairframe}{\save {"3,3"+<10pt,0pt>}\PATH
    ~={**\dir{-}} '"2,3"+<10pt,0pt>
                  `u^l
                  '"2,2"+<20pt,10pt>
                  `l_u
                  '"1,2"+<10pt,0pt>
                  `u^l
                  '"1,1"+<0pt,10pt>
                  `l^d
                  '"3,1"+<-10pt,0pt>
                  `d^r
                  '"3,3"+<0pt,-10pt>
                  `r^u
  \restore}
\newcommand{\SmallSquareLeftBottomFrame}{\save {"3,2"+<10pt,0pt>}\PATH
    ~={**\dir{-}} '"2,2"+<10pt,0pt>
                  `u^l
                  '"2,1"+<0pt,10pt>
                  `l^d
                  '"3,1"+<-10pt,0pt>
                  `d^r
                  '"3,2"+<0pt,-10pt>
                  `r^u
  \restore}

\newcommand{\SquareFrame}{\save {"3,3"+<10pt,0pt>}\PATH
    ~={**\dir{-}} '"1,3"+<10pt,0pt>
                  `u^l
                  '"1,1"+<0pt,10pt>
                  `l^d
                  '"3,1"+<-10pt,0pt>
                  `d^r
                  '"3,3"+<0pt,-10pt>
                  `r^u
  \restore}

\newcommand{\obj}[1]{#1}

\newcommand{\s}[1]{\mbox{#1}}
\newcommand{\eqcl}[1]{\ensuremath{<{#1}>}}

\newcommand{\secondcomparecd}{\xymatrix{
    k \ar@{=}[r]
        & k\\
    \tilde{k} \ar[u]^{-\tau} \ar@<0.5ex>[r]^{\eta_i} \ar@<-0.5ex>[r]
        &k\oplus(a'\oplus\tilde{p}) \ar[u]^{p_l}
        \ar@<0.5ex>[r]^{\mu_i} \ar@<-0.5ex>[r]
            &a'\oplus p\ar@{=}[d]\\
    q \ar[u]^{\omega} \ar[r]^{0\oplus\nu}&
        a'\oplus\tilde{p} \ar[u]^{i_r} \ar[r]^{1\oplus\chi}
            & a'\oplus p \\
    \save{"1,2"+<0pt,10pt>}\PATH~={**\dir{-}}
          '"1,1"+<0pt,10pt>
          `l^d
          '"3,1"+<-10pt,0pt>
          `d^r
          `r^u
          '"1,1"+<10pt,-20pt>
          `u_r
          '"1,2"+<0pt,-10pt>
          `r^u
          `u^l
    \restore
    }
}
\newcommand{\PHIcd}[5]{
  \xymatrix{
    {#1}\ar@<0.5ex>[r]\ar@<-0.5ex>[r]
        & {#2}\ar@<0.5ex>[r]\ar@<-0.5ex>[r]
          & {#3}\\
    {#1} \ar@{=}[u] \ar[r]
        & {#1}\oplus{#4} \ar@<0.5ex>[u]\ar@<-0.5ex>[u]\ar[r]
          & {#4} \ar[u]\\
        & {#5}\ar@<0.5ex>[u]\ar@<-0.5ex>[u]\ar@{=}[r]
          & {#5} \ar[u]\\
  }
}
\newcommand{\phiconstr}[1][]{\xymatrix{
    \obj{a'}\ar@<0.5ex>[r]^{f_i}\ar@<-0.5ex>[r]&
            \obj{a}\ar@<0.5ex>[r]^{g_i}\ar@<-0.5ex>[r]&
                 \obj{a''}\\\
    \obj{a'}\ar[r]^{i_l}\ar@{=}[u] &
            \obj{a'}\oplus\obj{p_{#1}}\ar[r]^{p_r}\ar@<0.5ex>[u]^{\beta_i}\ar@<-0.5ex>[u]
                & \obj{p}_{#1}\ar[u]^{\rho}\\
    &
            \obj{k}_{#1}\ar@{=}[r]\ar@<-0.5ex>[u]\ar@<0.5ex>[u]^{\alpha_i} &
         \obj{k}_{#1}\ar[u]^{\sigma} }
} 
\newcommand{\phiconstrtilde}[1][]{\xymatrix{
    \obj{a'}\ar@<0.5ex>[r]^{f_i}\ar@<-0.5ex>[r]&
            \obj{a}\ar@<0.5ex>[r]^{g_i}\ar@<-0.5ex>[r]&
                 \obj{a''}\\
    \obj{a'}\ar[r]^{i_l}\ar@{=}[u] &
            \obj{a'}\oplus\obj{\tilde{p}_{#1}}\ar[r]^{p_r}\ar@<0.5ex>[u]^{\tilde{\beta}_i}\ar@<-0.5ex>[u]
                & \obj{\tilde{p}}_{#1}\ar[u]^{\tilde{\rho}} \\
    &
            \obj{\tilde{k}}_{#1}\ar@{=}[r]\ar@<-0.5ex>[u]\ar@<0.5ex>[u]^{\tilde{\alpha}_i} &
         \obj{\tilde{k}}_{#1}\ar[u]^{\tilde{\sigma}} }
} 

\newcommand{\phiconstrcompare}{\xymatrix{
    k\ar@<0.5ex>[r]^{\alpha_i}\ar@<-0.5ex>[r]&
              \obj{a'}\oplus\obj{p}\ar@<0.5ex>[r]^{\beta_i}\ar@<-0.5ex>[r]&
                   \obj{a}\\
    \obj{k}\ar[r]^{i_l}\ar@{=}[u] &
               \obj{k}\oplus(\obj{a'}\oplus\obj{\tilde{p}})\ar[r]^{p_r}\ar@<0.5ex>[u]^{\mu_i}\ar@<-0.5ex>[u]
                    & \obj{a'}\oplus\obj{\tilde{p}}\ar@<-0.5ex>[u]
                    \ar@<0.5ex>[u]^{\tilde{\beta}_ i}\\
    &\obj{\tilde{k}}\ar@{=}[r]\ar@<-0.5ex>[u]\ar@<0.5ex>[u]^{\eta_i} &
                    \obj{\tilde{k}}\ar@<-0.5ex>[u]\ar@<0.5ex>[u]^{\tilde{\alpha}_i}\\
    \Lframe
    }
} 

\begin{document}

\title{An Algebraic Proof of \\
 Quillen's Resolution Theorem for $K_1$}

\author{Ben Whale\footnote{ben.whale@anu.edu.au}
  \footnote{Centre for Gravitational Physics,
  Department of Physics,
  Faculty of Science,
  The Australian National University,
  Canberra,
  ACT 0200,
  Australia
  }}

\maketitle

\begin{abstract}
  In his 1973 paper \cite{quillen} Quillen proved a resolution theorem for the
  $K$-Theory of an exact category; his proof was
  homotopic in nature. By using the main result of Nenashev's paper
  \cite{nenashevonetoone}, we are able to give an
  algebraic proof of Quillen's Resolution Theorem for $K_1$ of an
  exact category.  This represents an advance towards the goal of
  giving an essentially algebraic subject an algebraic
  foundation.
  
  \emph{Mathematics Subject Classification:} 18F25, 19B99

  \emph{Keywords:} Resolution Theorem, $K$-Theory, Exact Category
\end{abstract}

\section*{Introduction}
This paper presents an algebraic proof of Quillen's Resolution
Theorem of $K_1$ of an exact category.  The original proof of
Quillen's Resolution Theorem, \cite{quillen} uses homotopic
techniques. This research was done with the eventual aim of giving
algebraic proofs of most of the major homotopic results in the
area of $K$-Theory of exact categories. This would provide a new,
and hopefully insightful way, to work in the field.

The paper is divided into three sections. The first reviews the
results necessary for this work from Nenashev's three papers,
\cite{nenashevonto,nenashevall} and \cite{nenashevonetoone} as
well as making an important observation about the proof of his
main result. The second section presents Quillen's Resolution
Theorem and gives the algebraic proof for the case of $K_1$ of an
exact category.  We leave the proof of a lemma to the third and
last section as it is long, computational and distracts from the
main result.

This result was produced as the research component of a
BSc(Hons) under the supervision of Amnon Neeman. The research was
done at the Mathematical Sciences Institute of the ANU. Funding was
provided by the ANU in the form of an ANU Honours Scholarship.
Thank you to Amnon for his suggestions and support, with which
this research would not have been completed.

\section{Nenashev and $K_1$ of an Exact Category}\label{Nenashev.cha}
  Between 1996 and 1998 Nenashev published three papers
\cite{nenashevonto,nenashevall} and \cite{nenashevonetoone}. He
showed how it was possible to construct a group $D(\cat{M})$ from
an exact category $\cat{M}$ such that there was an isomorphism
$\xymatrix@1{D(\cat{M})\ar[r]^m & K_1(\cat{M})}$.\
We begin by giving an overview of the material that we shall need
from these papers.

\subsection{Definitions}

\begin{Def}
  Let $\cat{M}$ be an exact category.
  A \dses\ in $\cat{M}$, is two \ses s on the same objects, $a',a,a''\in\cat{M}$. That is, a \dses,
\[
  \DSES{a'}{a}{a''}{f_i}{}{g_i}{}
\]
 where $i=1,2$, is really two \ses s,
\[
  \SESM{a'}{a}{a''}{f_1}{g_1}
\]
and
\[
  \SESM{a'}{a}{a''}{f_2}{g_2}.
\]
\end{Def}

\begin{Def}
  A \comdia\ is a diagram
\[
  \xymatrix{
    &0&0&0\\
    0\ar[r]&
      a'\UP[r]^{f_i^a}\DW[r]\ar[u]
        & a \UP[r]^{g_i^a}\DW[r]\ar[u]
          & a''\ar[u]\ar[r] &0\\
    0\ar[r]&
      b'\UP[r]^{f_i^b}\DW[r]\UP[u]^{\beta'_i}\DW[u]
       & b\UP[r]^{g_i^b}\DW[r]\UP[u]^{\beta_i}\DW[u] & b''\UP[u]^{\beta''_i}\DW[u]\ar[r] &0\\
    0\ar[r]&
      c'\UP[r]^{f_i^c}\DW[r]\UP[u]^{\alpha'_i}\DW[u]
         & c\UP[r]^{g_i^c}\DW[r]\UP[u]^{\alpha_i}\DW[u]& c''\UP[u]^{\alpha''_i}\DW[u]\ar[r] &0\\
    &0\ar[u]&0\ar[u]&0\ar[u]
    }
\]
where each row and column is a \dses\ such that morphisms of the same subscript commute. Where no
subscript is given we shall assume that right maps commute with top
maps and left maps with bottom maps. When writing a \comdia\ we shall leave out the $0$ objects.
\end{Def}

\begin{Def}
Given an exact category $\cat{M}$, let $D(\cat{M})$ be the group with
generators \eqcl{d} for all \dses s $d$ in ${\cat{M}}$, and with the relations,
\begin{enumerate}
  \item Given $\eqcl{d} = \eqcl{\DSES{a'}{a}{a''}{f_i}{}{g_i}{}}$
        if $f_1=f_2$ and $g_1=g_2$ then $\eqcl{d}=0$.\label{rel1}
  \item If there exists a \comdia,\label{rel2}
        \[
          \Comdia
        \]
        where,
        \begin{gather*}
          \eqcl{H_T}=\eqcl{\DSE{a'}{a}{a''}}\\
          \eqcl{H_M}=\eqcl{\DSE{b'}{b}{b''}}\\
          \eqcl{H_B}=\eqcl{\DSE{c'}{c}{c''}}\\
          \eqcl{V_L}=\eqcl{\DSE{c'}{b'}{a'}}\\
          \eqcl{V_M}=\eqcl{\DSE{c}{b}{a}}\\
          \eqcl{V_R}=\eqcl{\DSE{c''}{b''}{a''}}
        \end{gather*}
        then we have the relation,
        \[
          \eqcl{H_T}-\eqcl{H_M}+\eqcl{H_B}=\eqcl{V_L}-\eqcl{V_M}+\eqcl{V_R}
        \]
        in $D(\cat{M})$.
\end{enumerate}
\end{Def}

\subsection{A Result about a Double Short Exact Sequence in $D(\cat{M})$}

This result was first proved by Nenashev in \cite{nenashevonetoone}.
We will use it implicitly throughout the rest of this paper.

\begin{Lem}\label{nen.lem.pl}
  Consider the \dses\
  \[
    d = \xymatrix{
        0\ar[r]&a\UP[r]^-{i_r}\DW[r]_-{i_l}
          &a\oplus a\UP[r]^-{-p_l}\DW[r]_-{p_r}
            & a\ar[r]
              &0
      }
  \]
  then $\eqcl{d}=0$ in $D(\cat{M})$.
\end{Lem}

\subsection{Nenashev's Isomorphism}

This theorem is the foundation for our proof of Quillen's Resolution
Theorem as it allows us to work with the elements of $K_1$ algebraically.

\begin{Thm}\label{thm.nenashev}
  For any exact category $\cat{M}$, there exists an isomorphism
  $$\xymatrix@1{m:D(\cat{M})\ar[r] & K_1(\cat{M})}$$
\end{Thm}
\begin{proof}
 Refer to the papers \cite{nenashevonto,nenashevall} and \cite{nenashevonetoone}.
\end{proof}

\subsection{The Generating Relations of $D(\cat{M})$}\label{subsec.gen}

In \cite{nenashevonetoone} Nenashev constructs an inverse to the
group homomorphism $m$. When he does this he only uses a few
\comdia s, giving certain relations. We can conclude that every relation
in $D(\cat{M})$ is generated by the few \comdia s that he used.
After some careful observation the reader will find that each of relations 
in \cite{nenashevonetoone} are given by one of the three following diagrams.
  \[
    \xymatrix{
      a' \UP[r]^{f_i'}\DW[r]
        & a \UP[r]^{g_i'}\DW[r]
          & a''\\
      a' \ar@{=}[u]\UP[r]^{f_i}\DW[r]
       & a\UP[r]^{g_i}\DW[r]\UP[u]^{\alpha_i}\DW[u]
         & a''\ar@{=}[u] \\
      0 \ar[u] & 0 \ar[u] & 0 \ar[u]\\
    }\quad\quad\quad\quad\quad\quad
    \xymatrix{
        & a\ar@{=}[r]
          & a\\
      b' \ar[r]^{f^b}
        & b \ar[r]^{g^b}\UP[u]^{\beta_i}\DW[u]
          & b''\UP[u]^{\beta_i''}\DW[u]\\
      b' \UP[u]^{\alpha_i'}\DW[u]\ar[r]^{f^c}
        & c \UP[u]^{\alpha_i}\DW[u]\ar[r]^{g^c}
          & c''\UP[u]^{\alpha_i''}\DW[u]
    }
  \]
and
  \[
    \xymatrix@C=1.3cm@R=1.1cm{
     a \UP[r]^{f_2}\DW[r]_{f_1}
       &  b \UP[r]^{g_2}\DW[r]_{g_1}
         &c \\
     a\oplus a \UP[u]^{-p_l}\DW[u]_{p_r}\UP[r]^{f_1\oplus f_2}
     \DW[r]_{f_1\oplus f_2}
       & b\oplus b \UP[u]^{-p_l}\DW[u]_{p_r} \UP[r]^{g_1\oplus g_2}
       \DW[r]_{g_1\oplus g_2}
         & c\oplus c \UP[u]^{-p_l}\DW[u]_{p_r} \\
     a \UP[r]^{f_1}\DW[r]_{f_2}\UP[u]^{i_r}\DW[u]_{i_l}
       &  b \UP[r]^{g_1}\DW[r]_{g_2}\UP[u]^{i_r}\DW[u]_{i_l}
         &c \UP[u]^{i_r}\DW[u]_{i_l}\\
    }
  \]

This observation shall
be important in our proof of the Resolution Theorem.

\section{The Proof of Quillen's Resolution Theorem for $K_1$ using
Nenashev's Isomorphism}\label{proof.cha}
  We shall prove Quillen's Resolution Theorem for $K_1$ of an exact
category using the algebraic `description' of $K_1$ given by
theorem \ref{thm.nenashev}. Specifically, we shall prove the
following theorem:

\begin{Thm}[Quillen's Resolution Theorem for $K_1$]\label{thm.resolution}
Let \s{\cat{M}} be an exact category and \s{\cat{F}} a full subcategory of \s{\cat{M}}
such that for all \ses s \[\SES{a'}{a}{a''},\]
\begin{enumerate}
  \item If $a',a''\in{\cat{F}}$ then
        $a\in{\cat{F},}$\label{prop3}
  \item If $a\in\s{\cat{F}}$ then
        $a'\in\s{\cat{F},}$\label{prop1}
  \item For any $a'' \in \s{\cat{M}}$ there exists a \ses s, as above, so that
        $a \in \s{\cat{F}.}$ \label{prop2}
\end{enumerate}
Then the inclusion functor $\xymatrix@1{\cat{F}\ar@{^{(}->}!<9pt,0pt>;[r]&\cat{M}}$ induces an
isomorphism
$$\xymatrix{i_*:K_1(\cat{F})\ar[r]&K_1(\cat{M})}.$$
\end{Thm}
Quillen's original proof of this theorem may be found in \cite{quillen}.

\subsubsection*{A Few Remarks}

For the rest of this section we shall assume that \cat{M} and
\cat{F} satisfy the hypotheses of the Resolution Theorem.
We shall use results from homological algebra, such as the snake
lemma, throughout the rest of this paper. This is justified by the
Gabriel-Quillen Embedding Theorems \cite[pp 399]{gabriel}.
Also, we shall draw a ring around the objects in a \comdia\
if they are known
to be in
\cat{F}.

\subsection{Definitions, Some Useful Groups and their Relationships}\label{sec.nen.def}

\begin{Def}\label{intro.def2}
We shall say that a \dses,
\[
  \DSE{a'}{a}{a''}
\]
is of type 0 if there are no restrictions on $a',a$ and $a''$, type 1 if
$a'\in{\cat{F}}$, type 2 if $a',a\in{\cat{F}}$ and of
type 3 if $a',a,a''\in{\cat{F}}$.
\end{Def}

\begin{Def}\label{intro.def4}
For all $j=0,1,2,3$ let $F_j$ be the free abelian group with generators $[d]_j$,
where $d$ is a \dses\ of type $j$. Where unambiguous we shall drop the
$j$ and write $[d]$.
\end{Def}

From definition \ref{intro.def2},
we see that we have inclusion homomorphisms between the $F_j$,
$
  \xymatrix{F_3\ar@{^{(}->}[r]^{i_3}
              &F_2\ar@{^{(}->}[r]^{i_2}
                &F_1\ar@{^{(}->}[r]^{i_1}
                  &F_0
  }
$.

\begin{Def}\label{def.T}
  For all $j=0,1,2,3$, let $T_j$ be the quotient of $F_j$ by the relation
  \[
    [\DSES{a'}{a}{a''}{f}{f}{g}{g}]_j=0
  \]
  and the relations
 given by
  the \comdia s below.

  \case{1} For $j=0$ we include all relations
  $$[V_L]-{[V_M]}+{[V_R]}={[H_T]}-{[H_M]}+{[H_B]}$$
  given by specializations of the \comdia
  \[
     \Comdia
  \]
  Thus $T_0\simeq K_1(\cat{M})$, by theorem \ref{thm.nenashev}.

  \case{2} For $j=1$ we restrict our relations to those given by
    specializations of the following \comdia s; for convenience we write
    under each \comdia\ the relation that it gives in $T_1$.
  \[
  \begin{gathered}
    \xymatrix{
      a'\ar@{=}[d]\UP[r]\DW[r]
        & a\ar@{=}[d]\UP[r]\DW[r]
          &a''\\
      a'\UP[r]\DW[r]
        &a\UP[r]\DW[r]
          &a''\UP[u]\DW[u]\\
      0\ar[u] & 0\ar[u] &0\ar[u]\\ \Lframe
    }\\
    [H_T]-[H_M]=[V_R]
  \end{gathered}\quad\quad\quad\quad
  \begin{gathered}
    \xymatrix@C=1.3cm@R=1cm{
      a'\UP[r]^{f_2}\DW[r]_{f_1}
       &a\UP[r]^{g_2}\DW[r]_{g_1}
         &a''\\
      a'\oplus a'\ar[r]^{f_1\oplus f_2}
      \DW[u]_{p_r}\UP[u]^{-p_l}
        &a\oplus a\ar[r]^{g_1\oplus g_2}
        \DW[u]_{p_r}\UP[u]^{-p_l}
          &a''\oplus a''\DW[u]_{p_r}\UP[u]^{-p_l}\\
      a'\UP[r]^{f_1}\DW[r]_{f_2}\DW[u]_{i_l}\UP[u]^{i_r}
       &a\UP[r]^{g_1}\DW[r]_{g_2}\DW[u]_{i_l}\UP[u]^{i_r}
         &a''\DW[u]_{i_l}\UP[u]^{i_r}\\
      \save {"1,1"+<-22pt,0pt>}\PATH~={**\dir{-}}
            '"3,1"+<-22pt,0pt>
            `d^r
            '"3,1"+<9pt,-10pt>
            `r^u
            '"1,1"+<19pt,0pt>
            `u^l
            '"1,1"+<-12pt,10pt>
            `l^d
      \restore
     }\\
     [H_T]+[H_B]=0
  \end{gathered}
  \]
  \newline
  \[
  \begin{gathered}
    \xymatrix{
      a'\ar[r]
        &a\ar[r]
          &a''\\
      a'\UP[u]\DW[u]\ar[r]
        &b\UP[u]\DW[u]\ar[r]
          &b''\UP[u]\DW[u]\\
        &c\UP[u]\DW[u]\ar@{=}[r]
          &c\UP[u]\DW[u]\\
      \save {"3,2"+<0pt,-10pt>}\PATH~={**\dir{-}}
            '"3,3"+<0pt,-10pt>
            `r^u
            `u^l
            '"3,2"+<0pt,10pt>
            `l^d
            `d^r
      \restore
    }\\
    [V_L]-[V_M]+[V_R]=0
  \end{gathered}\quad\quad\quad\quad
  \begin{gathered}
    \xymatrix@C=1cm{
        a'\UP[r]\DW[r]
            &a\UP[r]\DW[r]
            &a''\\
        a'\ar@{=}[u]\ar[r]^-{\Vmat{1}{0}}
            &a'\oplus p\UP[u]\DW[u]\ar[r]^(.55){\Hmat{0}{1}}
            &p\UP[u]\DW[u]\\
            &k\UP[u]\DW[u]\ar@{=}[r]
            &k\UP[u]\DW[u]\\
        \save {"3,3"+<10pt,0pt>}\PATH~={**\dir{-}}
                '"2,3"+<10pt,8pt>
                `u^l
                '"2,1"+<20pt,18pt>
                `l_u
                '"1,1"+<10pt,0pt>
                `u^l
                `l^d
                '"3,1"+<-10pt,0pt>
                `d^r
                '"3,3"+<0pt,-10pt>
                `r^u
        \restore
    }\\
    [H_T]=[V_R]-[V_M]
  \end{gathered}
  \]\newline

  \case{3} If $j=2$ we restrict our relations to those given by
    specializations of the following \comdia s; for convenience we write
    under each \comdia\ the relation that it gives in $T_2$.
  
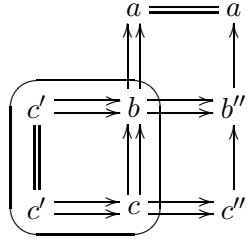
\[
  \begin{gathered}
    \xymatrix{
        & a\ar@{=}[r]
          &a\\
      c'\ar@{=}[d]\UP[r]\DW[r]
        & b\UP[u]\DW[u]\UP[r]\DW[r]
          & b''\ar[u]\\
      c'\UP[r]\DW[r]
        &c\UP[r]\DW[r]\UP[u]\DW[u]
          &c''\ar[u]\\ \SmallSquareLeftBottomFrame
    }\\
    [H_B]-[H_M]=-[V_M]
  \end{gathered}\quad\quad\quad\quad
  
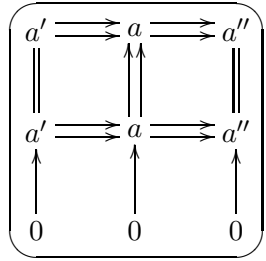
\begin{gathered}
    \xymatrix{
      a'\ar@{=}[d]\UP[r]\DW[r]
        & a\UP[r]\DW[r]
          &a''\ar@{=}[d]\\
      a'\UP[r]\DW[r]
        &a\UP[r]\DW[r]\UP[u]\DW[u]
          &a''\\
      0\ar[u] & 0\ar[u] &0\ar[u]\\ \SquareFrame
    }\\
    [H_T]-[H_M]=-[V_M]
  \end{gathered}\quad\quad\quad\quad
  \]\newline
  
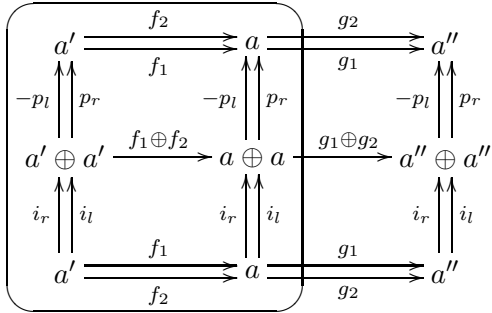
\[
  \begin{gathered}
    \xymatrix@C=1.3cm@R=1cm{
      a'\UP[r]^{f_2}\DW[r]_{f_1}
       &a\UP[r]^{g_2}\DW[r]_{g_1}
         &a''\\
      a'\oplus a'\ar[r]^{f_1\oplus f_2}
      \DW[u]_{p_r}\UP[u]^{-p_l}
        &a\oplus a\ar[r]^{g_1\oplus g_2}
        \DW[u]_{p_r}\UP[u]^{-p_l}
          &a''\oplus a''\DW[u]_{p_r}\UP[u]^{-p_l}\\
      a'\UP[r]^{f_1}\DW[r]_{f_2}\DW[u]_{i_l}\UP[u]^{i_r}
       &a\UP[r]^{g_1}\DW[r]_{g_2}\DW[u]_{i_l}\UP[u]^{i_r}
         &a''\DW[u]_{i_l}\UP[u]^{i_r}\\
      \save{"1,1"+<-22pt,5pt>}\PATH~={**\dir{-}}
           '"3,1"+<-22pt,-5pt>
           `d^r
           '"3,2"+<9pt,-15pt>
           `r^u
           '"1,2"+<19pt,5pt>
           `u^l
           '"1,1"+<-12pt,15pt>
           `l^d
      \restore
     }\\
     [H_T]+[H_B]=0
  \end{gathered}\quad\quad\quad\quad
  
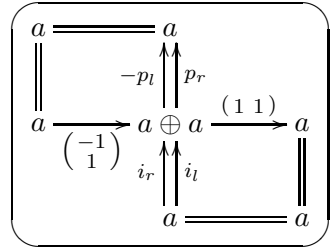
\begin{gathered}
    \xymatrix@C=1cm{
     a\ar@{=}[r]
       & a &\\
     a\ar[r]_-{\left(\begin{smallmatrix}-1\\1\end{smallmatrix}\right)}
     \ar@{=}[u]
       &a\oplus
       a\UP[u]^{-p_l}\DW[u]_{p_r}\ar[r]^-{\left(\begin{smallmatrix}1&1\end{smallmatrix}\right)}
         &a\ar@{=}[d]\\\
       &a\ar@{=}[r]\UP[u]^{i_r}\DW[u]_{i_l}
         &a\\ \SquareFrame
  }\\
  [V_M]=0
  \end{gathered}
  \]\newline
  
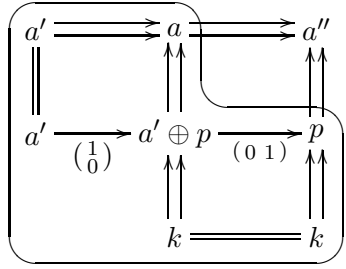
\[
  \begin{gathered}
    \xymatrix@C=1cm{
      a'\UP[r]\DW[r]
        &a\UP[r]\DW[r]
          &a''\\
      a'\ar@{=}[u]\ar[r]_-{\Vmat{1}{0}}
        &a'\oplus p\UP[u]\DW[u]\ar[r]_-{\Hmat{0}{1}}
          &p\UP[u]\DW[u]\\
        &k\UP[u]\DW[u]\ar@{=}[r]
          &k\UP[u]\DW[u]\\ \Chairframe
    }\\
    [H_T]=[V_R]-[V_M]
  \end{gathered}
  \]\newline

  \case{4} If $j=3$ then we allow all relations
  $${[V_L]}-{[V_M]}+{[V_R]}={[H_T]}-{[H_M]}+{[H_B]}$$
  given by specializations of the \comdia
  \[
    \xymatrix{
    a'\UP[r]\DW[r] & a \UP[r]\DW[r] & a''\\
    b'\UP[r]\DW[r]\UP[u]\DW[u] & b\UP[r]\DW[r]\UP[u]\DW[u]& b''\UP[u]\DW[u]\\
    c'\UP[r]\DW[r]\UP[u]\DW[u] & c\UP[r]\DW[r]\UP[u]\DW[u]& c''\UP[u]\DW[u]
    \SquareFrame
    }
  \]
  Thus $T_3\simeq K_1(\cat{F})$, by theorem \ref{thm.nenashev}.

   In all cases we shall denote the equivalence class of $[d]_j$ by
   $\eqcl{d}_j$. When unambiguous we shall drop the $j$ and simply
   write $\eqcl{d}$. From the definition of $T_j$ it is clear that there exists
		surjective group homomorphisms $\xymatrix{{\theta_j}:F_j\ar[r]&T_j,}$
		for all $j=0,1,2,3$.
\end{Def}

\begin{Def}\label{intro.def3}
We shall say that a \comdia\ $D$ is of type $j$ if it is a
specialization of one
of the diagrams given in the definition of $T_j$.
\end{Def}

\subsubsection*{Lemma \ref{nen.lem.pl} revisited}

The \comdia\ used in the proof of lemma \ref{nen.lem.pl} is
of type $j$, thus we can conclude that
$
   \eqcl{\xymatrix{
        0\ar[r]&a\UP[r]^-{i_r}\DW[r]_-{i_l}
          &a\oplus a\UP[r]^-{-p_l}\DW[r]_-{p_r}
            & a\ar[r]
              &0
      }}_j=0.
$

\subsection{The Group Homomorphisms $\phi_{j}$}

In order to prove the Resolution Theorem we shall construct a
group homomorphism that is an inverse to $i_*$.  To do this we
show how we may construct homomorphisms
$\xymatrix{{\phi_{j+1}}:F_j\ar[r]&F_{j+1}}$ that induce functions
$\phi_{j+1}^*$ between $T_j$ and $T_{j+1}$.  Before we do this, we present a construction that will allow us to
define $\phi_{j+1}$.

\subsubsection*{The \Phicons}
\begin{Cons}\label{consp_d}
  For all \dses s $d$
  \[
    \DSES{a'}{a}{a''}{f_i}{}{g_i}{}
  \]
  there exists $p\in\cat{F}$ such that we have the commutative
  triangle
  \[
    \xymatrix{
      p \ON[rr]^{g_i\eta_i}\UPON[dr]\DWON[dr]_{\eta_i}&& a''\\
      & a\UPON[ur]\DWON[ur]_{g_i}
    }
  \]
\end{Cons}
\begin{proof}
We construct the pullback of the diagram
\[
  \xymatrix{ a \ON[r]^{g_1} & a'' \\ & a \ON[u]^{g_2}}
\]
to get the object $a\times_{a''} a$. By property
(\ref{prop2}) of theorem \ref{thm.resolution}, we can find $p\in\cat{F}$ and
$\xymatrix@1{{\psi:p}\ON[r]& a\times_{a''} a}$, an admissible
epimorphism. This gives us the commutative
diagram
\[
  \xymatrix{& a \ON[r]^{g_1} & a'' \\
            & a\times_{a''} a \ON[u]^{\gamma_1}\ON[r]^{\gamma_2} & a
              \ON[u]^{g_2}\\
            {p} \ON[ur]^{\psi}}
\]
Let $\eta_i=\gamma_i\psi$. Then we have the commutative triangle
\[
  \xymatrix{
      p \ON[rr]^{g_i\eta_i}\UPON[dr]^{\eta_i}\DWON[dr]&& a''\\
      & a\UPON[ur]\DWON[ur]_{g_i}
  }
\]
as required.
\end{proof}

\begin{Cons}[The \Phicons]
  Given a type $j$ \dses\ $d$
  \[
    \DSES{a'}{a}{a''}{f_i}{}{g_i}{}
  \]
  and $p\in \cat{F}$ such that we have the commutative
  triangle
  \[
    \xymatrix{
      p \ON[rr]^{g_i\eta_i}\UP[dr]\DWON[dr]_{\eta_i}&& a''\\
      & a\UPON[ur]\DWON[ur]_{g_i}
    }
  \]
  we can construct a type $j+1$ \dses, $\phi(p,d)$, such that $\eqcl{d}_j=-\eqcl{\phi(p,d)}_j$.
\end{Cons}

\begin{proof}
The commutative triangle allows us to form the \comdia\
\begin{gather*}\label{intro.dia1}
  \xymatrix@C=1cm{
    a'\DW[r]\UP[r]^{f_i}
      & a\DW[r]\UP[r]^{g_i}
        & a'' \\
    a' \ar@{=}[u]\ar[r]^-{\Vmat{1}{0}}
      & a'\oplus p \DW[u]\UP[u]^{\omega_i^{a''}}\ar[r]^-{\Hmat{0}{1}}
        & p\ar[u]^{g_i\eta_i}\\
      & k\DW[u]\UP[u]^{\nu_i^{a''}}\ar@{=}[r]
        & k \ar[u]^{\tau}\\
  }
\end{gather*}
where $\omega_i^{a''}=(f_i, \eta_i)$.
By property (\ref{prop1}) of theorem \ref{thm.resolution} we see that
$k\in \cat{F}$, thus the \dses\ $V_M$ 
is of type $j+1$. We shall denote this \dses\ by $\phi(p,d)$. Note
that the \comdia, above, is of type $j$ thus we have the relation,
$\eqcl{d}_j=-\eqcl{\phi(p,d)}_j$ as required.
\end{proof}

We shall often wish to apply construction \ref{consp_d} to a \dses\ $d$,
$
  \DSES{a'}{a}{a''}{f_i}{}{g_i}{}
$
and then use the object $p\in\cat{F}$, given by construction
\ref{consp_d}, to apply the \Phicons\ to $d$. When we do this we
shall simply say that we have formed the \Phicons\ over the
morphisms $g_i$.

\subsubsection*{The Definition of $\phi_j$}

As a result of construction \ref{consp_d}, for each \dses\ $d$ we may choose $p\in\cat{F}$ so that the
\Phicons\ may be applied to $d$ using $p$. We shall denote this $p$ by
$p_d$.

\begin{Def}
  Define $\xymatrix{{\phi_{j+1}}:F_j\ar[r] & F_{j+1}}$ by
  $\phi_{j+1}([d]_j)=[\phi(p_d,d)]_{j+1}$ then extend by
  linearity to the whole group, $F_j$.
\end{Def}

  \subsection{The Proof of The Resolution Theorem}

So far we have constructed a number of groups and group
homomorphisms so that we have the diagram,
\[
  \xymatrix@!C{
    F_3\ar@{^{(}->}[r]^{i_3}
       \ar[d]^{\theta_3}
      &F_2\ar@{^{(}->}[r]^{i_2}
          \ar[d]^{\theta_2}
          \ar@/^/[l]^{\phi_3}
        &F_1\ar@{^{(}->}[r]^{i_1}
            \ar[d]^{\theta_1}
            \ar@/^/[l]^{\phi_2}
          &F_0\ar[d]^{\theta_0}
              \ar@/^/[l]^{\phi_1}\\
    T_3\simeq K_1(F)\ar@(dl,dr)[rrr]_{i_*}
      &T_2
        &T_1
          &T_0\simeq K_1(M)
  }
\]

It is possible to see that there are a number of relations
between the homomorphisms above. We shall only need two of these
relations for our proof.

\begin{Lem}\label{proofintro.lem}
   For all $j=1,2,3$, let $\theta_j,\phi_j,i_j$ and $\theta_0$ be defined as
   above. Then we have the two equations,
   \begin{enumerate}
     \item \label{intro.item3}
           $\theta_3 \phi_3\phi_2\phi_1 i_1
           i_2 i_3=-\theta_3$,
     \item \label{intro.item6}$\theta_0 i_1 i_2 i_3
            \phi_3 \phi_2 \phi_1=-\theta_0$.
   \end{enumerate}
\end{Lem}
\begin{proof}
   The proof follows from the fact that $\eqcl{\phi(p_d,d)}_{j}=-\eqcl{d}_j$.
\end{proof}

The next lemma is the result that our proof rests on. The proof is
long and computational in nature, hence we leave it to the next
section.

\begin{MnLem}\label{lemmas.cor}
  The functions $\xymatrix{{\phi_{j+1}:F_j}\ar[r]&F_{j+1}}$ induce
  functions $\xymatrix{{\phi_{j+1}^*:T_j}\ar[r]&T_{j+1}}$ such that
  the diagram
\begin{gather}\label{induced.dia}
  \xymatrix@!C{
    F_3
       \ar[d]^{\theta_3}
      &F_2
          \ar[d]^{\theta_2}
          \ar[l]^{\phi_3}
        &F_1
            \ar[d]^{\theta_1}
            \ar[l]^{\phi_2}
          &F_0\ar[d]^{\theta_0}
              \ar[l]^{\phi_1}\\
    T_3\simeq K_1(F)
      &T_2\ar[l]^-{\phi_3^*}
        &T_1\ar[l]^{\phi_2^*}
          &T_0\simeq K_1(M)\ar[l]^(.6){\phi_1^*}
  }
\end{gather}
commutes.
\end{MnLem}

\begin{Def}
  Define $\xymatrix{{\varphi}:K_1(M)\ar[r] & K_1(F)}$, by $\varphi=\phi_3^*\phi_2^*\phi_1^*$.
\end{Def}

We can now prove that the Resolution Theorem follows from Lemma \ref{lemmas.cor}.

\begin{proof}[Proof of Theorem \ref{thm.resolution}]
  We observe that
    $\varphi\theta_0=\theta_3\phi_3\phi_2\phi_1$
  by the commutativity of diagram (\ref{induced.dia}) and that
    $i^*\theta_3=\theta_0i_1i_2i_3$
  trivially. Hence by lemma \ref{proofintro.lem}) we know that
  $i^*\varphi\theta_0=\theta_0i_1i_2i_3\phi_3\phi_2\phi_1=-\theta_0$ and that
  $\varphi i^*\theta_3=\theta_3\phi_3\phi_2\phi_1i_1i_2i_3=-\theta_3$.
  Therefore as $\theta_0$ and $\theta_3$ are both epimorphisms we know that
  $i^*\varphi=\varphi i^*=-1_{K_1(M)}$. Thus, $-\varphi$ is an inverse to $i^*$, and so the inclusion functor
  $\xymatrix{{i^*}:K_1(F)\ar[r]& K_1(M)}$ is an isomorphism, as required.
\end{proof}

  \section{The Proof of the Key Lemma}

\subsection{Preliminary Results}

In preparation for the results needed to prove the key lemma, we present five
constructions.

\begin{Cons}\label{cons1}
   Given a specialization of the \comdia\ $D$
\[
  \xymatrix@C=1cm@R=1cm{
      & a \ar@{=}[r] & a \\
    b'\UP[r]^{f^b_i}\DW[r]
      & b \UP[r]^{g^b_i}\DW[r]\UP[u]^{\beta_i}\DW[u]
        & b''\UP[u]^{\beta_i''}\DW[u]\\
   b'\DW[r]\UP[r]^{f^c_i}\UP[u]^{\alpha_i'}\DW[u]
      & c \DW[r]\UP[r]^{g^c_i}\UP[u]^{\alpha_i}\DW[u]
        & c''\UP[u]^{\alpha_i''}\DW[u]
  }
\]
we may construct a specialization of the \comdia\ $D'$
\[
  \xymatrix@!C{
    b'\UP[r]^{f^b_i}\DW[r]
      & b\UP[r]^{g^b_i}\DW[r]
        & b''\\
    b'\UP[u]\DW[u]\UP[r]^-{\Vmat{f_i^c}{0}}\DW[r]
      & c\oplus p\UP[u]\DW[u]\UP[r]^-{\Smat{g_i^c}{0}{0}{1}}\DW[r]
        &c''\oplus p\UP[u]\DW[u]\\
     & k\UP[u]\DW[u]\ar@{=}[r]
        &k\UP[u]\DW[u]
      \save {"3,2"+<0pt,-10pt>}\PATH~={**\dir{-}}
            '"3,3"+<0pt,-10pt>
            `r^u
            `u^l
            '"3,2"+<0pt,10pt>
            `l^d
            `d^r
      \restore
  }
\]
where $V_M'=\phi(p,V_M)$ and $V_R'=\phi(p,V_R)$.
\end{Cons}
\begin{proof}
  We take the \Phicons\ over the two maps $\beta_i$, which allows us to construct the two
  commutative triangles
\[
  \xymatrix{
     p\ON[rr]^{\beta_i\eta_i}\UPON[dr]^{\eta_i}\DWON[dr]&&a\\
     &  b\UPON[ur]^{\beta_i}\DWON[ur]
  }\quad\quad\quad\quad\quad\quad
  \xymatrix{
     p\ON[rr]^{\beta_i''g^b_i\eta_i}\UPON[dr]^{g^b_i\eta_i}\DWON[dr]&&a\\
     &  b''\UPON[ur]\DWON[ur]_{\beta_i''}
  }
\]
Letting $k=\mbox{ker}(\beta_i\eta_i)$ we find the \comdia\ $D'$
\[
  \xymatrix@C=1.1cm{
    b'\UP[r]^{f^b_i}\DW[r]
      & b\UP[r]^{g^b_i}\DW[r]
        & b''\\
    b'\UP[u]\DW[u]\UP[r]^-{\Vmat{f_i^c}{0}}\DW[r]
      & c\oplus p\UP[u]\DW[u]\UP[r]^-{\Smat{g_i^c}{0}{0}{1}}\DW[r]
        &c''\oplus p\UP[u]\DW[u]\\
     & k\UP[u]\DW[u]\ar@{=}[r]
        &k\UP[u]\DW[u]
      \save {"3,2"+<0pt,-10pt>}\PATH~={**\dir{-}}
            '"3,3"+<0pt,-10pt>
            `r^u
            `u^l
            '"3,2"+<0pt,10pt>
            `l^d
            `d^r
      \restore
  }
\]
where $V_M'=\phi(p,V_M)$ and $V_R'=\phi(p,V_R)$ as required.
\end{proof}

\begin{Cons}\label{cons2}
Given a specialization of the \comdia\ $D$
\[
  \xymatrix@=1cm{
    a'\ar[r]^{f^a}
      & a\ar[r]^{g^a}
        & a''\\
    a'\UP[u]^{\beta_i'}\DW[u]\UP[r]^{f^b_i}\DW[r]
      & b\UP[u]^{\beta_i}\DW[u]\UP[r]^{g^b_i}\DW[r]
        &b''\UP[u]^{\beta_i''}\DW[u]\\
      & c\UP[u]^{\alpha_i}\DW[u]\ar@{=}[r]
        & c\UP[u]^{\alpha_i''}\DW[u]
  }
\]we may construct a specialization of the \comdia\ $D'$
\[
  \xymatrix{
      & b''\ar@{=}[r] & b''\\
      k^a \UP[r]\DW[r]\ar@{=}[d]
        &c\oplus p
        \UP[u]\DW[u]\UP[r]\DW[r]
          &b\UP[u]^{g^b_i}\DW[u]\\
      k^a \ar[r]
        & k^{a''}\UP[r]\DW[r]\UP[u]\DW[u]
          & a'\UP[u]^{f^b_i}\DW[u] \SmallLBLframe
    }
\]
where $H_M'=\phi(p,V_M)$ and $V_M'=\phi(p,V_R)$ and if $V_L$ is of type $j$ then,
\[
  \eqcl{H_B'}_{j+1}=\eqcl{\phi(k^{a''},V_L)}_{j+1}.
\]
\end{Cons}
\begin{proof}
  We form the \Phicons\ over the maps $\beta_i$, to get the two commutative
  triangles
  
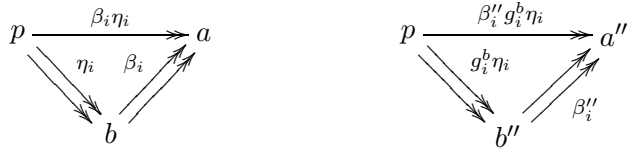
\[
    \xymatrix{
      p\ON[rr]^{\beta_i\eta_i}\UPON[dr]^{\eta_i}\DWON[dr]&&a\\
      & b\UPON[ur]^{\beta_i}\DWON[ur]
    }\quad\quad\quad\quad\quad\quad
    \xymatrix{
      p\ON[rr]^{\beta_i''g^b_i\eta_i}\UPON[dr]^{g^b_i\eta_i}\DWON[dr]&&a''\\
      & b''\UPON[ur]\DWON[ur]_{\beta_i''}
    }
  \]

  We can construct the \comdia
  \[
    \xymatrix@C=1cm@R=1cm{
      & b''\ar@{=}[r] & b''\\
      k^a \UP[r]^-{\nu_i^a}\DW[r]\ar@{=}[d]
        &c\oplus p
        \UP[u]^{\omega_i^{a''}}\DW[u]\UP[r]^-{\omega_i^{a}}\DW[r]
          &b\UP[u]^{g^b_i}\DW[u]\\
      k^a \UP[r]^{\mu_i}\DW[r]
        & k^{a''}\UP[r]^{\chi_i}\DW[r]\UP[u]^{\nu_i^{a''}}\DW[u]
          & a'\UP[u]^{f^b_i}\DW[u]
      \save {"3,2"+<10pt,3pt>}\PATH~={**\dir{-}}
             `u"3,1"
             '"3,1"+<20pt,13pt>
             `l"2,1"+<10pt,0pt>
             '"2,1"+<10pt,0pt>
             `u^l
             `l^d
             '"3,1"+<-10pt,0pt>
             `d^r
             '"3,2"+<0pt,-10pt>
             `r^u
             '"3,2"+<10pt,3pt>
    \restore
    }
  \]
  By considering, however, the projections onto $p$
  we see that $\mu_1=\mu_2$. Let $\mu=\mu_i$.  We now need to compute the map $\chi_i$.
  Note that we have the two commutative diagrams, with short exact
  rows and columns
  
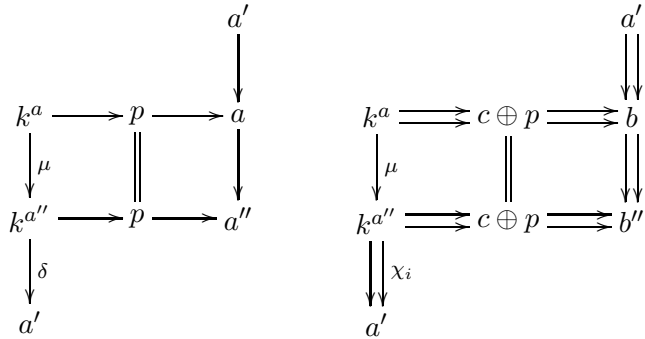
\[
    \xymatrix{
      && a'\ar[d]\\
      k^a\ar[r]\ar[d]^{\mu}
        &p\ar[r]
          & a\ar[d]\\
      k^{a''}\ar[r]\ar[d]^{\delta}
        &p\ar[r]\ar@{=}[u]
          & a''\\
      a'
    }\quad\quad\quad
    \xymatrix{
      && a'\UP[d]\DW[d]\\
      k^a\UP[r]\DW[r]\ar[d]^{\mu}
        &c\oplus p\UP[r]\DW[r]
          & b\UP[d]\DW[d]\\
      k^{a''}\UP[r]\DW[r]\UP[d]^{\chi_i}\DW[d]
        &c\oplus p\UP[r]\DW[r]\ar@{=}[u]
          & b''\\
      a'
    }
  \]
  The right diagram maps onto the left diagram, by the obvious
  maps. These diagrams give us two commutative squares
  from which it is possible to see that $\chi_i=\beta_i^{'-1}\delta$.
  Therefore we have the commutative
  triangle
  
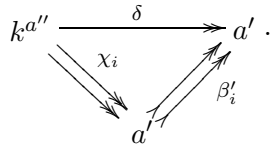
\[
    \xymatrix{
      k^{a''}\ON[rr]^{\delta}\DWON[dr]\UPON[dr]^{\chi_i}&&a'\\
      & a'\ar@<0.5ex>@{>->>}!<8pt,8pt>;[ur]
      \ar@<-0.5ex>@{>->>}!<8pt,8pt>;[ur]_{\beta'_i}
    }.
  \]
  If $V_L$ is of type $j$ then
  the \dses\ $\phi(k^{a''},V_L)$
  is of type $j+1$ and we have the type $j+1$
  \comdia
  \[
    \xymatrix@C=1cm{
    {0}\ar[r]
        & {a'}\ar@{=}[r]
          & {a'}\\
    {0} \ar@{=}[u] \ar[r]^-{\Vmat{1}{0}}
        & {0}\oplus{k^{a''}} \ar@<0.5ex>[u]\ar@<-0.5ex>[u]_{\Hmat{0}{\chi_i}}\ar[r]_-{\Hmat{0}{1}}
          & {k^{a''}} \ar@<0.5ex>[u]\ar@<-0.5ex>[u]_{\chi_i}\\
        & {k^a}\ar[u]^-{\Vmat{0}{\mu}}\ar@{=}[r]
          & {k^a} \ar[u]_{\mu}
    }
  \]
  which gives us the relation
    $\eqcl{H'_B}_{j+1}=\eqcl{\phi(k^{a''},V_L)}_{j+1}$
  as required.
\end{proof}

  \begin{Cons}\label{cons2.1}
  Given a specialization of the \comdia\ $D$
  \[
    \xymatrix@C=1cm{
      a'\UP[r]^{f_2}\DW[r]_{f_1}
       &a\UP[r]^{g_2}\DW[r]_{g_1}
         &a''\\
      a'\oplus a'\ar[r]^-{f_1\oplus f_2}
      \DW[u]_{p_r}\UP[u]^{-p_l}
        &a\oplus a\ar[r]^-{g_1\oplus g_2}
        \DW[u]_{p_r}\UP[u]^{-p_l}
          &a''\oplus a''\DW[u]_{p_r}\UP[u]^{-p_l}\\
      a'\UP[r]^{f_1}\DW[r]_{f_2}\DW[u]_{i_l}\UP[u]^{i_r}
       &a\UP[r]^{g_1}\DW[r]_{g_2}\DW[u]_{i_l}\UP[u]^{i_r}
         &a''\DW[u]_{i_l}\UP[u]^{i_r}
     }
  \]we may construct a specialization of the \comdia\ $D'$
\[
  \xymatrix@C=1.3cm@R=1cm{
      k\UP[r]^-{\nu_2}\DW[r]_-{\nu_1}
       &a'\oplus p\UP[r]^-{\omega_2}\DW[r]_-{\omega_1}
         &a\\
      k\oplus k\ar[r]^-{\nu_1\oplus \nu_2}
      \DW[u]_{p_r}\UP[u]^{-p_l}
        &a'\oplus p \oplus a'\oplus p\ar[r]^-{\omega_1\oplus \omega_2}
        \DW[u]_{p_r}\UP[u]^{-p_l}
          &a\oplus a\DW[u]_{p_r}\UP[u]^{-p_l}\\
      k\UP[r]^-{\nu_1}\DW[r]_-{\nu_2}\DW[u]_{i_l}\UP[u]^{i_r}
       &a'\oplus p\UP[r]^-{\omega_1}\DW[r]_-{\omega_2}\DW[u]_{i_l}\UP[u]^{i_r}
         &a\DW[u]_{i_l}\UP[u]^{i_r}
      \save {"1,1"+<0pt,12pt>}\PATH ~={**\dir{-}}
                  '"1,1"+<7pt,12pt>
                  `r_d
                  '"3,1"+<17pt,-2pt>
                  `d_l
                  '"3,1"+<-12pt,-12pt>
                  `l_u
                  '"1,1"+<-22pt,2pt>
                  `u_r
                  '"1,1"+<0pt,12pt>
  \restore
     }
\]
where $H_T' = \phi(p,H_T)$ and $H_B' = \phi(p,H_B)$.
\end{Cons}
\begin{proof}
  Form the \Phicons\ over the maps $g_i$ to get the commutative
  diagram
  \[
    \xymatrix{
      p \ON[rr]^{g_i\eta_i}\UPON[dr]^{\eta_i}\DWON[dr] && a''\\
      & a \UPON[ur]^{g_i}\DWON[ur]
    }
  \]

  We can then find the \comdia
  \[
   \xymatrix@C=1.3cm@R=1cm{
      k\UP[r]^-{\nu_2}\DW[r]_-{\nu_1}
       &a'\oplus p\UP[r]^-{\omega_2}\DW[r]_-{\omega_1}
         &a\\
      k\oplus k\ar[r]^-{\nu_1\oplus \nu_2}
      \DW[u]_{p_r}\UP[u]^{-p_l}
        &a'\oplus p \oplus a'\oplus p\ar[r]^-{\omega_1\oplus \omega_2}
        \DW[u]_{p_r}\UP[u]^{-p_l}
          &a\oplus a\DW[u]_{p_r}\UP[u]^{-p_l}\\
      k\UP[r]^-{\nu_1}\DW[r]_-{\nu_2}\DW[u]_{i_l}\UP[u]^{i_r}
       &a'\oplus p\UP[r]^-{\omega_1}\DW[r]_-{\omega_2}\DW[u]_{i_l}\UP[u]^{i_r}
         &a\DW[u]_{i_l}\UP[u]^{i_r}
      \save {"1,1"+<0pt,12pt>}\PATH ~={**\dir{-}}
                  '"1,1"+<7pt,12pt>
                  `r_d
                  '"3,1"+<17pt,-2pt>
                  `d_l
                  '"3,1"+<-12pt,-12pt>
                  `l_u
                  '"1,1"+<-22pt,2pt>
                  `u_r
                  '"1,1"+<0pt,12pt>
  \restore
     }
  \]
  as required.
\end{proof}
\begin{Cons}\label{cons.new}
  Given a \comdia\ which is a specialization of the \comdia\ $D$
  \[
    \xymatrix@C=1.1cm@R=1.1cm{
        & a\ar@{=}[r]
          &a\\
      c'\ar@{=}[d]\UP[r]^{f_i^b}\DW[r]
        & b\UP[u]^{\beta_i}\DW[u]\UP[r]^{g_i^b}\DW[r]
          & b''\ar[u]^{\beta''}\\
      c'\UP[r]^{f_i^c}\DW[r]
        &c\UP[r]^{g_i^c}\DW[r]\UP[u]^{\alpha_i}\DW[u]
          &c''\ar[u]^{\alpha''}
    }
  \]
  we may construct a specialization of the \comdia\ $D'$
  \[
    \xymatrix{
      k^a\UP[r]\DW[r]
        & c\oplus p\UP[r]\DW[r]
          &b\\
      k^a\ar@{=}[u]\ar[r]
        & k^a\oplus c' \oplus p\ar[r]\UP[u]\DW[u]
          &c' \oplus p\UP[u]\DW[u]\\
        & k^{b''}\ar@{=}[r]\UP[u]\DW[u]
          & k^{b''}\UP[u]\DW[u] \SmallRectLeftTopVerFrame\SmallRectRightBottomHorFrame
    }
  \]
  where $H_T'=\phi(p,V_M)$, $V_R'=\phi(p,H_M)$ and if $H_B$ if of type 2, then
  \[
    \eqcl{V_M'}_{3}=\eqcl{\phi(k^a,H_B)}_{3}.
  \]
\end{Cons}
\begin{proof}
  We form the \Phicons\ over the morphisms $g_i^b$, which allows us two construct the two 
  commutative triangles,
  
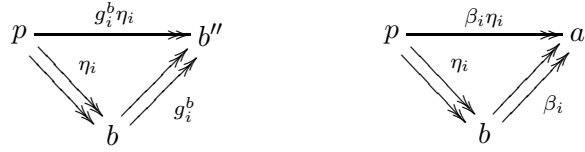
\[
    \xymatrix{
      p\ON[rr]^{g_i^b\eta_i}\UPON[dr]^{\eta_i}\DWON[dr] && b''\\
      & b\UPON[ur]\DWON[ur]_{g_i^b}
    }\quad\quad\quad\quad\quad
    \xymatrix{
      p\ON[rr]^{\beta_i\eta_i}\DWON[dr]\UPON[dr]^{\eta_i} && a\\
      & b\UPON[ur]\DWON[ur]_{\beta_i}
    }
  \]
  Thus we may also construct the \comdia
  \[
    \xymatrix@C=1.5cm@R=1cm{
      k^a\UP[r]^{\nu_i^a}\DW[r]
        & c\oplus p\UP[r]^{\omega_i^a}\DW[r]
          &b\\
      k^a\ar@{=}[u]\ar[r]
        & k^a\oplus c' \oplus p\ar[r]\UP[u]^{
        \left(\begin{smallmatrix}
        {p_L\nu_i^a}&{f_i^c}&{0}\\
        {p_R\nu_i^a}&{0}&{1}
        \end{smallmatrix}\right)}\DW[u]
          &c' \oplus p\UP[u]^{\omega_i^{b''}}\DW[u]\\
        & k^{b''}\ar@{=}[r]\UP[u]\DW[u]
          & k^{b''}\UP[u]^{\nu_i^{b''}}\DW[u] \SmallRectLeftTopVerFrame\SmallRectRightBottomHorFrame
    }
  \]

   Now suppose that $H_B$ is of type 2.
  We may construct the following
  commutative diagram
  
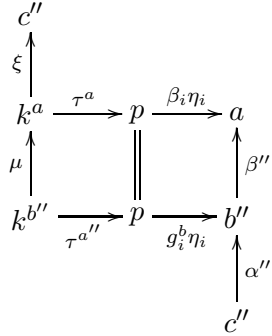
\[
    \xymatrix{
      {c''} \\
      k^a \ar[u]^{\xi}\ar[r]^{\tau^a} & p\ar[r]^{\beta_i\eta_i} & a\\
      k^{b''}\ar[r]_{\tau^{a''}}\ar[u]^{\mu} &
      p\ar@{=}[u]\ar[r]_{g^b_i\eta_i}&{b''}\ar[u]_{\beta''}\\
      &&{c''}\ar[u]_{\alpha''}
    }
  \]
  which gives us the relation $\alpha''\xi=g^b_i\eta_i\tau^a$.  Combined with
  $\omega^a_i\nu^a_i=0$ we can show that
  $\xi=-g_i^c p_L\nu_i^{a}$ and hence we know that we have the
  commutative triangle,
  
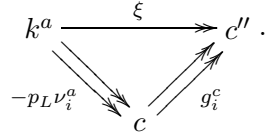
\[
    \xymatrix{
      k^a\ON[rr]^{\xi}\UPON[dr]\DWON[dr]_{-p_L\nu_i^{a}} && c''\\
      & c\UPON[ur]\DWON[ur]_{g_i^c}
    }.
  \]

  This triangle allows us to construct $\phi(k^{a},H_M)$
   and thus the two type
  3 \comdia s
  \[
    \xymatrix@R=1cm{
      k^{b''}\UP[r]\DW[r]
        & k^a\oplus c'\oplus p\UP[r]\DW[r]
          &c\oplus p\\
      k^{b''}\UP[r]\DW[r]\ar@{=}[u]
        & p\oplus c'\oplus
      k^a\UP[r]\DW[r]\ar[u]^{\left(\begin{smallmatrix}0&0&1\\0&1&0\\1&0&0
          \end{smallmatrix}\right)}
          &c\oplus p\ar@{=}[u]\\
      0\ar[u]&0\ar[u]&0\ar[u]
    \save {"3,3"+<16pt,0pt>}\PATH
    ~={**\dir{-}} '"1,3"+<16pt,0pt>
                  `u^l
                  '"1,1"+<-2pt,10pt>
                  `l^d
                  '"3,1"+<-12pt,0pt>
                  `d^r
                  '"3,3"+<6pt,-10pt>
                  `r^u
  \restore
    }\quad\quad
    \xymatrix@C=1.1cm{
      p\ar[r]_{\Vmat{0}{1}}
        &c\oplus p \ar[r]_-{\Hmat{1}{0}}
          &c\\
      p\ar@{=}[u]\ar[r]_(.4){\left(\begin{smallmatrix}
            1\\0\\0\\ \end{smallmatrix}\right)}
        &p\oplus c'\oplus k^a \UP[u]\DW[u]\ar[r]_-{\left(\begin{smallmatrix}
            0&1&0\\0&0&-1\end{smallmatrix}\right)}
          &c'\oplus k^a\UP[u]\DW[u]\\
        &k^{b''}\ar@{=}[r]\UP[u]\DW[u]
          &k^{b''}\UP[u]\DW[u]
    \save {"3,3"+<17pt,0pt>}\PATH
    ~={**\dir{-}} '"1,3"+<17pt,0pt>
                  `u^l
                  '"1,1"+<0pt,10pt>
                  `l^d
                  '"3,1"+<-10pt,0pt>
                  `d^r
                  '"3,3"+<7pt,-10pt>
                  `r^u
  \restore
    }
  \]
  give the relation $\eqcl{V_M'}_{3}=\eqcl{\phi(k^a,H_B)}_{3}$
  as required.
\end{proof}

\begin{Cons}\label{cons2.2}
  Given a specialization of the \comdia\ $D$
  \[
    \xymatrix@C=1.2cm{
      a'\UP[r]^{f_i}\DW[r]
        & a\UP[r]^{g_i}\DW[r]
          &a''\\
      a'\ar@{=}[u]\ar[r]_-{\Vmat{1}{0}}
        & a'\oplus p'\ar[r]^(.55){\Hmat{0}{1}}\UP[u]^{\Hmat{f_i}{\eta_i'}}\DW[u]
          &p'\UP[u]^{g_i\eta_i'}\DW[u]\\
        & k'\UP[u]^{\nu_i'}\DW[u]\ar@{=}[r]
          & k'\UP[u]^{\tau_i'}\DW[u]
          \SmallRectRightBottomHorFrame
  }
  \]
  we may construct a specialization of the \comdia\ $D'$
  \[
   \xymatrix@C=1cm{
      k' \UP[r]^-{\nu_i'}\DW[r]
        & a'\oplus p' \UP[r]^-{\Hmat{f_i}{\eta_i'}}\DW[r]
          & a\\
      k'\ar@{=}[u]\ar[r]^-{\left(\begin{smallmatrix}1\\0\\0\end{smallmatrix}\right)}
        & k'\oplus a'\oplus p\UP[u]\DW[u]
        \ar[r]^-{\left(\begin{smallmatrix}0&1&0\\0&0&1\\\end{smallmatrix}\right)}
          & a'\oplus p\UP[u]\DW[u]\\
        & k\UP[u]\DW[u]\ar@{=}[r]
          & k\UP[u]\DW[u] \SmallRectLeftTopVerFrame\SmallRectRightBottomHorFrame
    }
  \]
  where $V_R'=\phi(p,H_T)$
  and if $D$ is of type $j$ then
  \[
    \eqcl{V_M'}_{j+1} = \eqcl{\phi(p,V_R)}_{j+1}.
  \]
\end{Cons}
\begin{proof}
  Form the \Phicons\ over the maps $g_i\eta_i'$ to get
  the two commutative triangles,
  \[
    \xymatrix{
      p\ON[rr]^{g_i\eta_i'\gamma_i}\UPON[dr]\DWON[dr]_{\gamma_i} && a''\\
      & p'\UPON[ur]^{g_i\eta_i'}\DWON[ur]
    }\quad\quad\quad\quad\quad
    \xymatrix{
      p\ON[rr]^{g_i\eta_i'\gamma_i}\UP[dr]\DW[dr]_(.4){\eta_i'\gamma_i}
        && a''\\
      & a\UPON[ur]\DWON[ur]_{g_i}
    }
  \]
  We may construct the type $j+1$ \comdia
  \[
    \xymatrix@C=1cm{
      k' \UP[r]^-{\nu_i'}\DW[r]
        & a'\oplus p' \UP[r]^-{\Hmat{f_i}{\eta_i'}}\DW[r]
          & a\\
      k'\ar@{=}[u]\ar[r]^-{\left(\begin{smallmatrix}1\\0\\0\end{smallmatrix}\right)}
        & k'\oplus a'\oplus p\UP[u]\DW[u]
        \ar[r]^-{\left(\begin{smallmatrix}0&1&0\\0&0&1\\\end{smallmatrix}\right)}
          & a'\oplus p\UP[u]\DW[u]\\
        & k\UP[u]\DW[u]\ar@{=}[r]
          & k\UP[u]\DW[u] \SmallRectLeftTopVerFrame\SmallRectRightBottomHorFrame
    }
  \]

  Now, let $D$ be of type $j$. By using the first commutative
  triangle we can construct $\phi(p,V_R)$, since $\ker(\gamma_i)=\ker(g_i\eta_i')=k'$.
  Then we have the two type $j+1$ \comdia s
  \[
    \xymatrix@R=1cm{
      k\UP[r]\DW[r]
        &k'\oplus a'\oplus p\UP[r]\DW[r]
          & a'\oplus p'\\
      k\UP[r]\DW[r]\ar@{=}[u]
        &a'\oplus k'\oplus
        p\UP[r]\DW[r]\ar[u]^{\left(\begin{smallmatrix}
        0&1&0\\1&0&0\\0&0&1\end{smallmatrix}\right)}
          & a'\oplus p'\ar@{=}[u]\\
      0\ar[u]&0\ar[u]&0\ar[u]
      \SmallRectLeftTopVerFrame
    }\quad\quad
    \xymatrix@C=1.3cm{
      a'\ar[r]_-{\Vmat{1}{0}}
        &a'\oplus p'\ar[r]_-{\Hmat{0}{1}}
          &p'\\
      a'\ar@{=}[u]\ar[r]_-{\left(\begin{smallmatrix}1\\0\\0\end{smallmatrix}\right)}
        &a'\oplus k'\oplus
        p\ar[r]_-{\left(\begin{smallmatrix}0&1&0\\0&0&1\end{smallmatrix}\right)}
        \UP[u]\DW[u]
          &k'\oplus p\UP[u]\DW[u]\\
        &k\UP[u]\DW[u]\ar@{=}[r]
          &k\UP[u]\DW[u]
     \SmallRectRightBottomHorFrame
    }
  \]
  which together give us the relation
    $\eqcl{V_M'}_{j+1} = \eqcl{\phi(p,V_R)}_{j+1}$,
  as required.
\end{proof}

\begin{Cons}\label{cons2.3}
  Given a specialization of the \comdia\ $D$
  \[
    \xymatrix{
      a'\UP[r]^{f_i'}\DW[r]
        & a\UP[r]^{g_i'}\DW[r]
          &a''\\
      a'\UP[r]^{f_i}\DW[r]\UP[u]^{\beta_i'}\DW[u]
        &a\UP[r]^{g_i}\DW[r]\UP[u]^{\beta_i}\DW[u]
          &a''\UP[u]^{\beta_i''}\DW[u]\\
      0\ar[u] &0\ar[u]&0\ar[u]
      \save{"3,1"+<0pt,10pt>}\PATH~={**\dir{-}}
           '"3,3"+<0pt,10pt>
           `r_d
           `d_l
           '"3,1"+<0pt,-10pt>
           `l_u
           `u_r
      \restore
    }
  \]
  we may construct specializations of the two \comdia s $D_1$ and $D_2$ respectively
  
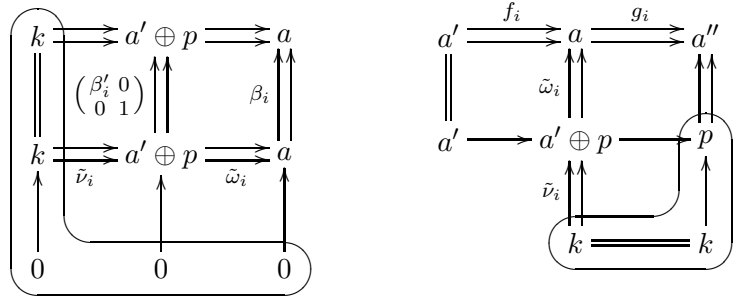
\[
    \xymatrix@R=1cm{
      k\UP[r]\DW[r]
        &a'\oplus p\UP[r]\DW[r]
        &a\\
      k\UP[r]\DW[r]_-{\tilde{\nu}_i}\ar@{=}[u]
        &a'\oplus p\UP[r]\DW[r]_-{\tilde{\omega}_i}\UP[u]^{\Smat{\beta_i'}{0}{0}{1}}\DW[u]
          &a\UP[u]^{\beta_i}\DW[u]\\
      0\ar[u]& 0\ar[u] &0\ar[u] \Lframe
    }\quad\quad\quad\quad
    \xymatrix{
      a'\UP[r]^{f_i}\DW[r]
        &a\UP[r]^{g_i}\DW[r]
          &a''\\
      a'\ar@{=}[u]\ar[r]
        &a'\oplus p\ar[r]\UP[u]^{\tilde{\omega}_i}\DW[u]
          &p\UP[u]\DW[u]\\
        &k\UP[u]^{\tilde{\nu}_i}\DW[u]\ar@{=}[r]
          &k\ar[u] \SmallRBLframe
    }
  \]
  where $H_T^1=\phi(p,H_T)$ and $V_R^2=\phi(p,V_R)$.
\end{Cons}
\begin{proof}
  We form the \Phicons\ over the maps $g_i'$, which gives us the three
  commutative triangles,
  
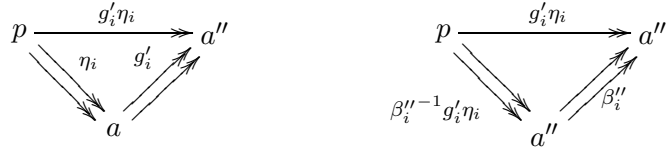
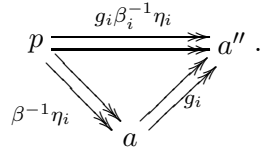
\[
    \xymatrix{
      p\ON[rr]^{g_i'\eta_i}\UPON[dr]^{\eta_i}\DWON[dr] && a''\\
      & a\UPON[ur]^{g_i'}\DWON[ur] &
    }\quad\quad\quad\quad\quad
    \xymatrix{
      p\ON[rr]^{g_i'\eta_i}\UPON[dr]\DWON[dr]_{{\beta_i''}^{-1}g_i'\eta_i} && a''\\
      & a''\UPON[ur]_{\beta_i''}\DWON[ur] &
    }
  \]
  and
  \[
  \xymatrix{
      p\UPON[rr]^{g_i\beta_i^{-1}\eta_i}\DWON[rr]\UPON[dr]\DWON[dr]_{\beta^{-1}\eta_i} && a''\\
      & a\UPON[ur]_{g_i}\DWON[ur] &
    }.
  \]
  Note that $\ker(g_i\beta_i^{-1}\eta_i)=\ker(g_i'\eta_i)$ as
  $g_i\beta_i^{-1}\eta_i$ differs from $g_i'\eta_i$ by the
  automorphism ${\beta_i''}^{-1}$. Hence from these triangles we form the \comdia s
  \[
    \xymatrix@R=1cm{
      k\UP[r]\DW[r]
        &a'\oplus p\UP[r]\DW[r]
        &a\\
      k\UP[r]\DW[r]_-{\tilde{\nu}_i}\ar@{=}[u]
        &a'\oplus p\UP[r]\DW[r]_-{\tilde{\omega}_i}\UP[u]^{\Smat{\beta_i'}{0}{0}{1}}\DW[u]
          &a\UP[u]^{\beta_i}\DW[u]\\
      0\ar[u]& 0\ar[u] &0\ar[u] \Lframe
    }\quad\quad
    \xymatrix{
      a'\UP[r]^{f_i}\DW[r]
        &a\UP[r]^{g_i}\DW[r]
          &a''\\
      a'\ar@{=}[u]\ar[r]
        &a'\oplus p\ar[r]\UP[u]^{\tilde{\omega}_i}\DW[u]
          &p\UP[u]\DW[u]_{g_i\beta_i^{-1}\eta_i}\\
        &k\UP[u]^{\tilde{\nu}_i}\DW[u]\ar@{=}[r]
          &k\ar[u] \SmallRBLframe
    }
  \]
  as required.
\end{proof}

  \subsection{The Homomorphisms $\theta_{j}\phi_j$ are Independent of Choice}

We now give a proof that, for all $j=1,2,3$, $\theta_{j}\phi_j$ is
independent of the choice of $p_d$. This result is used implicitly
throughout the rest of this paper.

\begin{Lem}\label{lem.results1:choice}
  Let $d=\xymatrix{0 \ar[r] & {a'} \ar@<0.5ex>[r]^{f_i}\ar@<-0.5ex>[r]
    & {a} \ar@<0.5ex>[r]^{g_i}\ar@<-0.5ex>[r] & {a''} \ar[r] &0}$ be the type $j$ \dses,
  and suppose we have the choice of $p$ or $p'$ for $p_d$. Then
  $\eqcl{\phi(p,d)}_{j+1}=\eqcl{\phi(p',d)}_{j+1}.$
\end{Lem}
\begin{proof}
Let $p,p'\in{F}$ be \st\ either are valid choices for $p_d$. Then
we have two commutative triangles,
\[
  \xymatrix{
    p\ON[rr]^{g_i\eta_i}\DW[dr]_{\eta_i}\UP[dr] && a''\\
    & a\UPON[ur]\DWON[ur]_{g_i}
  }\quad\quad\quad
\xymatrix{
    p'\ON[rr]^{g_i\eta_i'}\DW[dr]_{\eta_i'}\UP[dr] && a''\\
    & a\UPON[ur]\DWON[ur]_{g_i}
  }.
\]
We may form the pullback square
\[
  \xymatrix{
    p\ON[r]^{g_i\eta_i}&a''\\
    p\times_{a''}p'\ON[u]^{p_L}\ON[r]^-{p_R}
      & p'\ON[u]_{g_i\eta_i'}
  }
\]
and by property (\ref{prop3}) of theorem \ref{thm.resolution} we know that
$p\times_{a''}p'\in{F}$. For ease of notation, let $\tilde{p}=p\times_{a''}p'$.
We will show that $\eqcl{\phi(p,d)}_{j+1} =
\eqcl{\phi(\tilde{p},d)}_{j+1}$. Note that $\ker(p_L)=\ker(g_i\eta_i')$,
so again for ease of notation let $k'=\ker(p_L)$.

\case{1} Suppose that $j=0$.
We may construct the \comdia
\[
  \xymatrix@C=1.1cm{
      &a\ar@{=}[r]
        &a\\
    k'\ar[r]^-{\Vmat{0}{\gamma}}
      &a'\oplus{\tilde{p}}\UP[u]^{\tilde{\omega}_i}\DW[u]
      \ar[r]^-{\Smat{1}{0}{0}{p_L}}
        &a'\oplus p\UP[u]^{\omega_i}\DW[u]\\
    k'\ar@{=}[u]\DW[r]\UP[r]^{\mu_i}
      &{\tilde{k}}\DW[r]\UP[r]^{\chi_i}\UP[u]^{\tilde{\nu}_i}\DW[u]
        &k\UP[u]^{\nu_i}\DW[u]
    \save {"3,3"+<10pt,4pt>}\PATH
    ~={**\dir{-}} `u"3,1"
                  '"3,1"+<20pt,14pt>
                  `l"2,1"+<10pt,0pt>
                  '"2,1"+<10pt,0pt>
                  `u^l
                  `l^d
                  '"3,1"+<-10pt,0pt>
                  `d^r
                  '"3,3"+<0pt,-10pt>
                  `r^u
                  '"3,3"+<10pt,4pt>
    \restore
    }
\]
However by considering the projections onto $p$ and $\tilde{p}$ 
we see that $\chi_1=\chi_2$ and $\mu_1=\mu_2$. For consistency of notation, let
$\chi=\chi_i$ and $\mu=\mu_i$. Thus we
have really constructed the \comdia\ $D$
\[
  \xymatrix@C=1.1cm{
      &a\ar@{=}[r]
        &a\\
    k'\ar[r]^-{\Vmat{0}{\gamma}}
      &a'\oplus{\tilde{p}}\UP[u]^{\tilde{\omega}_i}\DW[u]
      \ar[r]^-{\Smat{1}{0}{0}{p_L}}
        &a'\oplus p\UP[u]^{\omega_i}\DW[u]\\
    k'\ar@{=}[u]\ar[r]^{\mu}
      &{\tilde{k}}\ar[r]^{\chi}\UP[u]^{\tilde{\nu}_i}\DW[u]
        &k\UP[u]^{\nu_i}\DW[u]
    \save {"3,3"+<10pt,3pt>}\PATH
    ~={**\dir{-}} `u"3,1"
                  '"3,1"+<20pt,13pt>
                  `l"2,1"+<10pt,0pt>
                  '"2,1"+<10pt,0pt>
                  `u^l
                  `l^d
                  '"3,1"+<-10pt,0pt>
                  `d^r
                  '"3,3"+<0pt,-10pt>
                  `r^u
                  '"3,3"+<10pt,3pt>
    \restore
    }
\]
We apply construction \ref{cons1} to $D$ to get the relation
  $\eqcl{\phi(p_1,V_M)}_1=\eqcl{\phi(p_1,V_R)}_1.$
As $V_M$ and $V_R$ are both of type 1 we know that 
$\eqcl{V_M}_1=\eqcl{V_R}_1$ and hence 
$\eqcl{\phi(\tilde{p},d)}_1=\eqcl{\phi(p,d)}_1$
since, by construction, $V_M=\phi(\tilde{p},d)$ and $V_R=\phi(p,d)$.

\case{2} Suppose that $j=1,2$.
We have the commutative triangle,
\[
  \xymatrix{
    a'\oplus\tilde{p} \DWON[rr]\UPON[rr]^{\tilde{\omega}_i}
    \ON[dr]_(.4){\Smat{1}{0}{0}{p_L}}&& a\\
    & a'\oplus p \UPON[ur]\DWON[ur]_{\omega_i}
  }
\]
with which we can construct the type $j+1$ \comdia\ $D$
\begin{gather*}
 \xymatrix@C=1.1cm@R=1.1cm{
    k\UP[r]\DW[r]_-{\nu_i}&
      {a'}\oplus{p}\UP[r]\DW[r]_-{\omega_i}&
        {a}\\
    {k}\ar[r]_-{\left(\begin{smallmatrix}1\\0\\0\end{smallmatrix}\right)}
    \ar@{=}[u] &
      {k}\oplus{a'}\oplus{\tilde{p}}\ar[r]_-{\left(\begin{smallmatrix}0&1&0\\0&0&1\end{smallmatrix}\right)}
      \UP[u]^{\xi_i}\DW[u]
        & {a'}\oplus{\tilde{p}}\DW[u]
        \UP[u]^{\tilde{\omega}_i}\\
      &{\tilde{k}}\ar@{=}[r]\DW[u]\UP[u]^{\rho_i} &
        {\tilde{k}}\DW[u]\UP[u]^{\tilde{\nu}_i}
    \save {"3,3"+<17pt,0pt>}\PATH
    ~={**\dir{-}} '"2,3"+<17pt,0pt>
                  `u^l
                  '"2,2"+<26pt,10pt>
                  `l_u
                  '"1,2"+<16pt,0pt>
                  `u^l
                  '"1,1"+<0pt,10pt>
                  `l^d
                  '"3,1"+<-10pt,0pt>
                  `d^r
                  '"3,3"+<7pt,-10pt>
                  `r^u
  \restore
  }
\end{gather*}
By considering the morphisms $\chi$ and $p_L$ we can construct the
type $j+1$ \comdia\
\[
  \xymatrix@C=1.1cm@R=1.1cm{
    k\ar@{=}[r]
      &k&\\
    {\tilde{k}}\ar[u]^-{\chi}\UP[r]^-{\rho_i}\DW[r]
      &k\oplus
      a'\oplus\tilde{p}\ar[u]^-{\left(\begin{smallmatrix}-1&0&0\end{smallmatrix}\right)}
      \UP[r]^-{\xi_i}\DW[r]
        &a'\oplus p\ar@{=}[d]\\
    k'\ar[u]^-{\mu}\ar[r]^(.35){\Vmat{0}{\gamma}}
      &a'\oplus\tilde{p}\ar[u]^-{\left(\begin{smallmatrix}0&0\\1&0\\0&1\end{smallmatrix}\right)}
      \ar[r]^-{\Smat{1}{0}{0}{p_l}}
        &a'\oplus p  \save {"3,3"+<19pt,-1pt>}\PATH
    ~={**\dir{-}} '"1,3"+<19pt,0pt>
                  `u^l
                  '"1,1"+<-2pt,10pt>
                  `l^d
                  '"3,1"+<-12pt,-1pt>
                  `d^r
                  '"3,3"+<9pt,-11pt>
                  `r^u
  \restore
  }
\]
These two \comdia s combined give us the relation
\[
	\eqcl{\phi(p,d)}_{j+1}=\eqcl{\phi(\tilde{p},d)}_{j+1}.
\]
Hence in either case we see that
$\eqcl{\phi(p,d)}_{j+1}=\eqcl{\phi(\tilde{p},d)}_{j+1}$ as was required.
\end{proof}

\begin{Cor}\label{cor.choice:independent}
  For all $j=1,2,3$, the function $\theta_j\phi_j$ is independent
  of the choice of $p_d$.
\end{Cor}
\begin{proof}
  Let $p$ and $p'$ be two choices for $p_d$. Then the lemma tells us
  that
  $\eqcl{\phi(p,d)}_{j+1}=\eqcl{\phi(p',d)}_{j+1}.$ Hence
  we have our result.
\end{proof}

Given a \dses\ $d$ we shall now denote\\ $\eqcl{\phi(p_d,d)}_j$ by
$\eqcl{\phi(d)}_j$, since $\eqcl{\phi(p_d,d)}_j$ is independent of
$p_d$.

  \subsection{The Functions $\phi_j$ Induce Group Homomorphisms}\label{sec.welldefined}

We now need to show that if 
  $\sum_{i=1}^{n} a_i\eqcl{d_i}_j=0, a_i\in\mathbb{Z}$
is a relation  then
  $\sum_{i=1}^{n} a_i\eqcl{\phi(d_i)}_{j+1}=0, a_i\in\mathbb{Z}$
is also a relation.  To do so we need only check this equation
for the generating relations of each
group.
  \subsubsection*{The Relations for $T_0$}

From the comments and diagrams in subsection \ref{subsec.gen}, we have three relations to check.

\begin{Lem}\label{lem.rel.1}
  Let $\eqcl{H_T}-\eqcl{H_M}=-\eqcl{V_M}$ be the relation given by the
  type 0
  \comdia\ $D$
\[
  \xymatrix{
      a' \UP[r]^{f_i'}\DW[r]
        & a \UP[r]^{g_i'}\DW[r]
          & a''\\
      a' \ar@{=}[u]\UP[r]^{f_i}\DW[r]
       & a\UP[r]^{g_i}\DW[r]\UP[u]^{\alpha_i}\DW[u]
         & a''\ar@{=}[u] \\
      0 \ar[u] & 0 \ar[u] & 0 \ar[u]
    \save  {"3,1"+<0pt,10pt>} \PATH ~={**\dir{-}}
                           '"3,3"+<0pt,10pt>
                           `r_d
                           `d_l
                           '"3,1"+<0pt,-10pt>
                           `l_u
                           `u_r
    \restore}
\]
  Then $\eqcl{\phi(H_T)}_1-\eqcl{\phi(H_M)}_1=-\eqcl{\phi(V_M)}_1$ is a
  relation in $T_1$.
\end{Lem}
\begin{proof}
  We apply construction \ref{cons2.3} to $D$ to get the relation
    $\eqcl{\phi(H_T)}_1-\eqcl{\phi(H_M)}_1=\eqcl{V_M}_1$,
  but we also know that $\eqcl{V_M}_1=-\eqcl{\phi(V_M)}_1$, and so we
  have the relation
  $\eqcl{\phi(H_T)}_1-\eqcl{\phi(H_M)}_1=-\eqcl{\phi(V_M)}_1,$ as required.
\end{proof}

\begin{Lem}
  Let $\eqcl{V_L}-\eqcl{V_M}+\eqcl{V_R}=0$ be the relation given by
  the type 0 \comdia\ $D$,
  \[
    \xymatrix@=1cm{
        & a\ar@{=}[r]
          & a\\
      b' \ar[r]^{f^b}
        & b \ar[r]^{g^b}\UP[u]^{\beta_i}\DW[u]
          & b''\UP[u]^{\beta_i''}\DW[u]\\
      b' \UP[u]^{\alpha_i'}\DW[u]\ar[r]^{f^c}
        & c \UP[u]^{\alpha_i}\DW[u]\ar[r]^{g^c}
          & c''\UP[u]^{\alpha_i''}\DW[u]
    }
  \]
  Then $\eqcl{\phi(V_L)}_1-\eqcl{\phi(V_M)}_1+\eqcl{\phi(V_R)}_1=0$ is a
  relation in $T_1$.
\end{Lem}
\begin{proof}
  We apply construction \ref{cons1} to $D$ and get the relation
    $\eqcl{V_L'}_1-\eqcl{\phi(V_M)}_1+\eqcl{\phi(V_R)}_1=0$
  and as we have to two type 1 \comdia s
  \[
    \xymatrix{
      b'\UP[r]\DW[r] &b'\ar[r]&0\\
      b'\ar@{=}[u]\ar[r]^{\Vmat{1}{0}}
        &b'\oplus 0\UP[u]\DW[u]\ar[r]^{\Hmat{0}{1}}&0\\
      0\ar[u]&0\ar[u]
     }\quad\quad\quad\quad\quad
    \xymatrix{
      0\ar[r] & b'\ar@{=}[r] &b'\\
      0\ar[r]& b'\oplus 0\UP[u]\DW[u]\ar[r]
        &b'\UP[u]\DW[u]\\
      & 0\ar[u]&0\ar[u]
    }
  \]
  we can see that $\eqcl{V_L'}=\eqcl{\phi(V_L)}$. Therefore we have the relation
    $\eqcl{\phi(V_L)}_1-\eqcl{\phi(V_M)}_1+\eqcl{\phi(V_R)}_1=0$
  as required.
\end{proof}

\begin{Lem}
  Let $\eqcl{H_T}+\eqcl{H_B}=0$ be the relation given by the type 0 \comdia\ $D$
  \[
    \xymatrix@C=1.5cm@R=1.2cm{
     a \UP[r]^{f_2}\DW[r]_{f_1}
       &  b \UP[r]^{g_2}\DW[r]_{g_1}
         &c \\
     a\oplus a \UP[u]^{-p_l}\DW[u]_{p_r}\UP[r]^-{f_1\oplus f_2}
     \DW[r]_{f_1\oplus f_2}
       & b\oplus b \UP[u]^{-p_l}\DW[u]_{p_r} \UP[r]^-{g_1\oplus g_2}
       \DW[r]_{g_1\oplus g_2}
         & c\oplus c \UP[u]^{-p_l}\DW[u]_{p_r} \\
     a \UP[r]^{f_1}\DW[r]_{f_2}\UP[u]^{i_r}\DW[u]_{i_l}
       &  b \UP[r]^{g_1}\DW[r]_{g_2}\UP[u]^{i_r}\DW[u]_{i_l}
         &c \UP[u]^{i_r}\DW[u]_{i_l}
    }
  \]
  Then $\eqcl{\phi(H_T)}_1+\eqcl{\phi(H_B)}_1=0$ is a relation in $T_1$.
\end{Lem}
\begin{proof}
  We apply construction \ref{cons2.1} to $D$ to get the relation
    $\eqcl{\phi(H_T)}_1+\eqcl{\phi(H_B)}_1=0$
  as required.
\end{proof}

\subsubsection*{The Relations for $T_1$}
For $T_1$ we need to check the four relations
given by the \comdia s in definition \ref{def.T}.

\begin{Lem}\label{lem.t1.1}
  Let $\eqcl{H_T}-\eqcl{H_M}=\eqcl{V_R}$ be the relation given by the
  type 1
  \comdia\ $D$
\[
  \xymatrix{
      a'\ar@{=}[d]\UP[r]^{f_i'}\DW[r]
        & a\ar@{=}[d]\UP[r]^{g_i'}\DW[r]
          &a''\\
      a'\UP[r]^{f_i}\DW[r]
        &a\UP[r]^{g_i}\DW[r]
          &a''\UP[u]\DW[u]\\
      0\ar[u] & 0\ar[u] &0\ar[u] \Lframe
    }
\]Then $\eqcl{\phi(H_T)}_2-\eqcl{\phi(H_M)}_2=\eqcl{\phi(V_R)}_2$ is
      a relation in $T_2$.
\end{Lem}
\begin{proof}
  Apply construction \ref{cons2.3} to $D$ to get two \comdia s $D'$ and $D''$ respectively
  
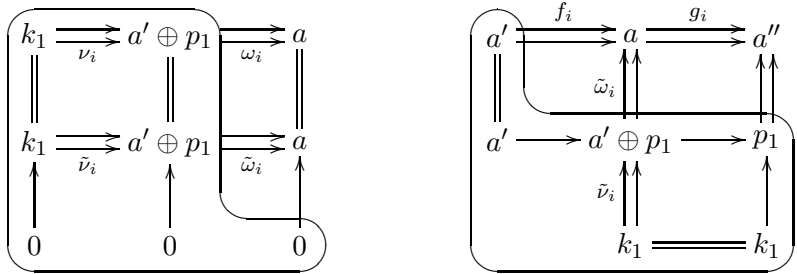
\[
    \xymatrix{
      k_1\UP[r]\DW[r]_-{\nu_i}\ar@{=}[d]
        & a'\oplus p_1 \UP[r]\DW[r]_-{\omega_i}\ar@{=}[d]
          & a\ar@{=}[d]\\
      k_1\UP[r]\DW[r]_-{\tilde{\nu}_i}
        & a'\oplus p_1 \UP[r]\DW[r]_-{\tilde{\omega}_i}
          & a\\
      0\ar[u] &0\ar[u]&0\ar[u]
      \save {"1,1"+<0pt,10pt>}\PATH ~={**\dir{-}}
                  '"1,2"+<9pt,10pt>
                  `r_d
                  '"3,2"+<19pt,20pt>
                  `d^r
                  '"3,3"+<0pt,10pt>
                  `r_d
                  `d_l
                  '"3,1"+<0pt,-10pt>
                  `l_u
                  '"1,1"+<-10pt,0pt>
                  `u_r
  \restore
    }\quad\quad\quad\quad\quad
    \xymatrix{
      a'\UP[r]^{f_i}\DW[r]
        &a\UP[r]^{g_i}\DW[r]
          &a''\\
      a'\ar@{=}[u]\ar[r]
        &a'\oplus p_1\ar[r]\UP[u]^{\tilde{\omega}_i}\DW[u]
          &p_1\UP[u]\DW[u]\\
        &k_1\UP[u]^{\tilde{\nu}_i}\DW[u]\ar@{=}[r]
          &k_1\ar[u] \MedChairframe
    }
  \]
  where $H_T'=\phi(H_T)$ and $V_R''=\phi(V_R)$. The \comdia\ $D'$ tells us that
  $\omega_i=\tilde{\omega}_i$ and $\nu_i=\tilde{\nu}_i$, so when we apply
  construction \ref{cons2.2} to $D''$ we get relation
    $\eqcl{\phi(H_T)}_2=\eqcl{\phi(H_M)}_2-\eqcl{\phi^2(V_R)}_2.$
  Now, as $\phi(V_R)$ is of type $2$
  we know that $\eqcl{\phi(V_R)}_2=-\eqcl{\phi^2(V_R)}_2.$ Hence we get the relation
    $\eqcl{\phi(H_T)}_2-\eqcl{\phi(H_M)}_2=\eqcl{\phi(V_R)}_2$
  as required.
\end{proof}

\begin{Lem}
Let $\eqcl{H_T}+\eqcl{H_B}=0$ be the relation given by the type 1
\comdia\ $D$
\[
  \xymatrix@C=1.1cm@R=1.1cm{
      a'\UP[r]^{f_2}\DW[r]_{f_1}
       &a\UP[r]^{g_2}\DW[r]_{g_1}
         &a''\\
      a'\oplus a'\ar[r]^-{f_1\oplus f_2}
      \DW[u]_{p_r}\UP[u]^{-p_l}
        &a\oplus a\ar[r]^-{g_1\oplus g_2}
        \DW[u]_{p_r}\UP[u]^{-p_l}
          &a''\oplus a''\DW[u]_{p_r}\UP[u]^{-p_l}\\
      a'\UP[r]^{f_1}\DW[r]_{f_2}\DW[u]_{i_l}\UP[u]^{i_r}
       &a\UP[r]^{g_1}\DW[r]_{g_2}\DW[u]_{i_l}\UP[u]^{i_r}
         &a''\DW[u]_{i_l}\UP[u]^{i_r}
      \save {"1,1"+<-21pt,0pt>}\PATH~={**\dir{-}}
            '"3,1"+<-21pt,0pt>
            `d^r
            '"3,1"+<9pt,-10pt>
            `r^u
            '"1,1"+<19pt,0pt>
            `u^l
            '"1,1"+<-11pt,10pt>
            `l^d
      \restore
     }
\]
Then $\eqcl{\phi(H_T)}_2+\eqcl{\phi(H_B)}_2=0$ is a relation in $T_2$.
\end{Lem}
\begin{proof}
  We apply construction \ref{cons2.1} to $D$ to get the relation
    $\eqcl{\phi(H_T)}_2+\eqcl{\phi(H_B)}_2=0$
  as required.
\end{proof}

\begin{Lem}
   Let $\eqcl{V_L}-\eqcl{V_M}+\eqcl{V_R}=0$ be the relation given
   by the type 1 \comdia\ $D$
   \[
     \xymatrix{
      a'\ar[r]
        &a\ar[r]
          &a''\\
      a'\UP[u]\DW[u]\ar[r]
        &b\UP[u]\DW[u]\ar[r]
          &b''\UP[u]\DW[u]\\
        &c\UP[u]\DW[u]\ar@{=}[r]
          &c\UP[u]\DW[u]
      \save {"3,2"+<0pt,-10pt>}\PATH~={**\dir{-}}
            '"3,3"+<0pt,-10pt>
            `r^u
            `u^l
            '"3,2"+<0pt,10pt>
            `l^d
            `d^r
      \restore
      }
    \]
    Then the relation
    $\eqcl{\phi(V_L)}_2-\eqcl{\phi(V_M)}_2+\eqcl{\phi(V_R)}_2=0$ is a
    relation in $T_2$.
\end{Lem}
\begin{proof}
    We apply construction \ref{cons2} to $D$ to get the relation
      $\eqcl{\phi(V_L)}_2-\eqcl{\phi(V_M)}_2+\eqcl{\phi(V_R)}_2=0$
    as required.
\end{proof}

\begin{Lem}
  Let $\eqcl{H_T}=\eqcl{V_R}-\eqcl{V_M}$ be the relation given
  by the type 1 \comdia\ $D$
  \[
    \xymatrix@C=1cm{
      a'\UP[r]\DW[r]
        &a\UP[r]\DW[r]
          &a''\\
      a'\ar@{=}[u]\ar[r]_-{\Vmat{1}{0}}
        &a'\oplus p_1\UP[u]\DW[u]\ar[r]_-{\Hmat{0}{1}}
          &p_1\UP[u]\DW[u]\\
        &k_1\UP[u]\DW[u]\ar@{=}[r]
          &k_1\UP[u]\DW[u] \MedChairframe
    }
  \]
  Then $\eqcl{\phi(H_T)}_2=\eqcl{\phi(V_R)}_2-\eqcl{\phi(V_M)}_2$ is a
  relation in $T_2$.
\end{Lem}
\begin{proof}
  We apply construction \ref{cons2.2} to $D$ to get the relation
    $\eqcl{V_M}_2=\eqcl{\phi(H_T)}_2-\eqcl{\phi(V_R)}_2.$
  However, since $V_M$ is of type $2$ we know that
  $\eqcl{V_M}_2=-\eqcl{\phi(V_M)}_2$. Thus we get the relation
  $\eqcl{\phi(H_T)}_2=\eqcl{\phi(V_R)}_2-\eqcl{\phi(V_M)}_2$ as required.
\end{proof}

\subsubsection*{The Relations for $T_2$}
For $T_2$ we need to check the relations
given by the \comdia s in definition \ref{def.T}.
\begin{Lem}
  Let $\eqcl{H_B}-\eqcl{H_M}=-\eqcl{V_M}$ be the relation given
  by the type 2 \comdia\ $D$
  \[
    \xymatrix{
        & a\ar@{=}[r]
          &a\\
      c'\ar@{=}[d]\UP[r]\DW[r]
        & b\UP[u]\DW[u]\UP[r]\DW[r]
          & b''\ar[u]\\
      c'\UP[r]\DW[r]
        &c\UP[r]\DW[r]\UP[u]\DW[u]
          &c''\ar[u] \SmallSquareLeftBottomFrame
    }
  \]
  Then  $\eqcl{\phi(H_B)}_3-\eqcl{\phi(H_M)}_3=-\eqcl{\phi(V_M)}_3$ is
  a relation in $T_3$.
\end{Lem}
\begin{proof}
  We apply construction \ref{cons.new} to $D$ to get the relation
    $\eqcl{\phi(V_M)}_3=\eqcl{\phi(H_M)}_3-\eqcl{\phi(H_B)}_3$
  as required.
\end{proof}

\begin{Lem}
  Let $\eqcl{H_T}-\eqcl{H_M}=-\eqcl{V_M}$ be the relation given
  by the type 2 \comdia\ $D$
  \[
    \xymatrix{
      a'\ar@{=}[d]\UP[r]\DW[r]
        & a\UP[r]\DW[r]
          &a''\ar@{=}[d]\\
      a'\UP[r]\DW[r]
        &a\UP[r]\DW[r]\UP[u]\DW[u]
          &a''\\
      0\ar[u] & 0\ar[u] &0\ar[u] \SquareFrame
    }
  \]
  Then $\eqcl{\phi(H_T)}_3-\eqcl{\phi(H_M)}_3=-\eqcl{\phi(V_M)}_3$ is
  a relation in $T_3$.
\end{Lem}
\begin{proof}
  Note that $D$ is also of type 3, hence we have the relation
    $\eqcl{H_T}_3-\eqcl{H_M}_3=-\eqcl{V_M}_3.$
  Also, $H_T,H_M$ and $V_M$ are all of type 3 so we know that
    $\eqcl{\phi(H_T)}_3-\eqcl{\phi(H_M)}_3=-\eqcl{\phi(V_M)}_3$
  as required.
\end{proof}

\begin{Lem}
  Let $\eqcl{H_T}+\eqcl{H_M}=0$ be the relation given by the type
  2 \comdia\ $D$
  \[
    \xymatrix@C=1.5cm@R=1.2cm{
      a'\UP[r]^{f_2}\DW[r]_{f_1}
       &a\UP[r]^{g_2}\DW[r]_{g_1}
         &a''\\
      a'\oplus a'\ar[r]^-{f_1\oplus f_2}
      \DW[u]_{p_r}\UP[u]^{-p_l}
        &a\oplus a\ar[r]^-{g_1\oplus g_2}
        \DW[u]_{p_r}\UP[u]^{-p_l}
          &a''\oplus a''\DW[u]_{p_r}\UP[u]^{-p_l}\\
      a'\UP[r]^{f_1}\DW[r]_{f_2}\DW[u]_{i_l}\UP[u]^{i_r}
       &a\UP[r]^{g_1}\DW[r]_{g_2}\DW[u]_{i_l}\UP[u]^{i_r}
         &a''\DW[u]_{i_l}\UP[u]^{i_r}
      \save {"1,1"+<-22pt,5pt>}\PATH~={**\dir{-}}
            '"3,1"+<-22pt,-5pt>
            `d^r
            '"3,2"+<7pt,-15pt>
            `r^u
            '"1,2"+<17pt,5pt>
            `u^l
            '"1,1"+<-12pt,15pt>
            `l^d
      \restore
     }
  \]
  Then $\eqcl{\phi(H_T)}_3+\eqcl{\phi(H_M)}_3=0$ is a relation in $T_3$.
\end{Lem}
\begin{proof}
  We apply construction \ref{cons2.1} to $D$ to get the relation
    $\eqcl{\phi(H_T)}_3+\eqcl{\phi(H_M)}_3=0$
  as required.
\end{proof}

\begin{Lem}
Let $\eqcl{V_M}=0$ be the relation given by the \comdia\ $D$
\[
  \xymatrix@C=1.1cm{
     a\ar@{=}[r]
       & a&\\
     a\ar[r]_-{\left(\begin{smallmatrix}-1\\1\end{smallmatrix}\right)}
     \ar@{=}[u]
       &a\oplus
       a\UP[u]^{-p_l}\DW[u]_{p_r}\ar[r]^-{\left(\begin{smallmatrix}1&1\end{smallmatrix}\right)}
         &a\ar@{=}[d]\\
       &a\ar@{=}[r]\UP[u]^{i_r}\DW[u]_{i_l}
         &a \SquareFrame
  }
\]
Then $\eqcl{\phi(V_M)}_3=0$ is a relation in $T_3$.
\end{Lem}
\begin{proof}
  The \comdia\ $D$ is also of type 3, therefore we have the
  relation $\eqcl{V_M}_3=0$.
  However, $V_M$ is of type 3 thus we know that
  $\eqcl{\phi(V_M)}_3=0$ as required.
\end{proof}

\begin{Lem}\label{lem.rel.end}
  Let $\eqcl{H_T}=\eqcl{V_R}-\eqcl{V_M}$ be the relation given
  by the type 2 \comdia\ $D$
  \[
    \xymatrix{
      a'\UP[r]\DW[r]
        &a\UP[r]\DW[r]
          &a''\\
      a'\ar@{=}[u]\ar[r]
        &a'\oplus p_1\UP[u]\DW[u]\ar[r]
          &p_1\UP[u]\DW[u]\\
        &k_1\UP[u]\DW[u]\ar@{=}[r]
          &k_1\UP[u]\DW[u] \Chairframe
    }
  \]
  Then
  $\eqcl{\phi(H_T)}_3=\eqcl{\phi(V_R)}_3-\eqcl{\phi(V_M)}_3$ is a
  relation in $T_3$.
\end{Lem}
\begin{proof}
  We apply construction \ref{cons2.2} to $D$ to get the relation
    $\eqcl{V_M}_3=\eqcl{\phi(H_T)}_3-\eqcl{\phi(V_R)}_3,$
  but $V_M$ is of type 3 and so we know that
    $\eqcl{\phi(V_R)}_3-\eqcl{\phi(V_M)}_3=\eqcl{\phi(H_T)}_3$
  as required.
\end{proof}

\subsubsection*{The Proof of the Key Lemma}

  The proof of the key lemma, lemma \ref{lemmas.cor}, follows from the
  comments at the beginning of this section
  and from lemmas \ref{lem.rel.1} to \ref{lem.rel.end}.

\bibliographystyle{plain}
\bibliography{Main}

\begin{thebibliography}{1}

\bibitem{nenashevonto}
A.~Nenashev.
\newblock {$K\sb 1$} by generators and relations.
\newblock {\em J. Pure Appl. Algebra}, 131(2):195--212, 1998.

\bibitem{nenashevall}
Alexander Nenashev.
\newblock Double short exact sequences produce all elements of {Q}uillen's
  {$K\sb 1$}.
\newblock In {\em Algebraic $K$-theory (Pozna\'n, 1995)}, volume 199 of {\em
  Contemp. Math.}, pages 151--160. Amer. Math. Soc., Providence, RI, 1996.

\bibitem{nenashevonetoone}
Alexander Nenashev.
\newblock Double short exact sequences and {$K\sb 1$} of an exact category.
\newblock {\em $K$-Theory}, 14(1):23--41, 1998.

\bibitem{quillen}
Daniel Quillen.
\newblock Higher algebraic {$K$}-theory. {I}.
\newblock In {\em Algebraic $K$-theory, I: Higher $K$-theories (Proc. Conf.,
  Battelle Memorial Inst., Seattle, Wash., 1972)}, pages 85--147. Lecture Notes
  in Math., Vol. 341. Springer, Berlin, 1973.

\bibitem{gabriel}
R.~W. Thomason and T.~Trobaugh.
\newblock Higher algebraic {$K$}-theory of schemes and of derived categories.
\newblock In {\em The Grothendieck Festschrift, Vol.\ III}, volume~88 of {\em
  Progr. Math.}, pages 247--435. Birkh\"auser Boston, Boston, MA, 1990.

\end{thebibliography}

\end{document}